\documentclass[a4paper,10pt]{article}
\usepackage[english,frenchb]{babel}
\usepackage[T1]{fontenc}
\usepackage[utf8]{inputenc}
\usepackage{amsmath}
\usepackage{array}
\usepackage{amsthm}
\usepackage{amssymb}
\input xy
\xyoption{all}

\newcommand{\A}{{\mathcal{A}}}
\newcommand{\B}{{\mathcal{B}}}
\newcommand{\C}{{\mathcal{C}}}
\newcommand{\D}{{\mathcal{D}}}
\newcommand{\F}{{\mathcal{F}}}
\newcommand{\Hh}{{\mathcal{H}}}
\newcommand{\Pp}{{\mathcal{P}}}
\newcommand{\E}{{\mathcal{E}}}

\newcommand{\s}{{\mathcal{S}}}

\newcommand{\K}{{\mathcal{K}}}

\newcommand{\col}{{\rm colim}\,}

\newcommand{\Pl}{{\mathbf{Pl}}}

\title{Sur l'homologie des groupes unitaires à coefficients polynomiaux}
\author{Aur\'elien DJAMENT}

\newtheorem{thi}{Th\'eor\`eme}

\newtheorem{thm}{Th\'eor\`eme}[section]
\newtheorem{pr}[thm]{Proposition}
\newtheorem{cor}[thm]{Corollaire}
\newtheorem{lm}[thm]{Lemme}

\theoremstyle{definition}
\newtheorem{defi}[thm]{D\'efinition}
\newtheorem{nota}[thm]{Notation}
\newtheorem{hyp}[thm]{Hypoth\`ese}

\theoremstyle{remark}
\newtheorem{rem}[thm]{Remarque}
\newtheorem{ex}[thm]{Exemple}

\begin{document}

\maketitle



\tableofcontents

\section*{Introduction}

\paragraph*{Présentation des principaux résultats}

Nous nous proposons de montrer que l'homologie des groupes unitaires sur un anneau $A$ (muni d'une anti-in\-vo\-lu\-tion), {\em à coefficients tordus par un foncteur polynomial} raisonnable (par exemple, une puissance tensorielle de la représentation tautologique), peut se calculer {\em stablement} à partir de l'homologie {\em à coefficients constants} des mêmes groupes et de groupes d'homologie des foncteurs que l'on peut déterminer explicitement dans les cas favorables. Il s'agit d'une généralisation des résultats obtenus précédemment par l'auteur avec C. Vespa dans l'article \cite{DV}, qui traite le cas où $A$ est un corps fini. Signalons aussi qu'A. Touzé a établi dans \cite{Touze} des résultats analogues pour les groupes orthogonaux sur les corps finis {\em vus comme groupes algébriques} ; dans le présent article, on ne traitera que de groupes {\em discrets}.

Précisons un petit peu les choses. Dans un premier temps, limitons-nous aux groupes unitaires {\em hyperboliques} $U_{n,n}(A)$. Sous des hypothèses raisonnables (à la Bass) sur $A$, l'homologie de ces groupes ne dépend plus de $n$ dès lors que cet entier est assez grand devant le degré homologique (cf. l'article \cite{MvdK} de Mirzaii et van der Kallen). Ceci est vrai non seulement pour l'homologie à coefficients constants, mais aussi pour l'homologie tordue par des coefficients polynomiaux. Pour éviter toute hypothèse sur $A$, nous nous intéresserons à la colimite $U_{\infty,\infty}(A)$ de ces groupes et à son homologie (c'est en ce sens qu'on parle d'homologie stable). Le calcul de cette homologie à coefficients constants constitue un problème extrêmement ardu, même lorsque $A$ est un corps, problème que nous n'abordons pas dans le présent travail. Les conjectures de Friedlander-Milnor (cf. \cite{Mil}) attestent de cette difficulté, mais des résultats partiels significatifs sont disponibles (voir notamment les articles \cite{K-an} et \cite{K-inv} de Karoubi, utilisant de puissantes méthodes de $K$-théorie). Si l'on se restreint aux corps de caractéristique nulle (l'un des cas les plus développés, pour les coefficients tordus, dans cet article), c'est d'ailleurs l'homologie {\em rationnelle} du groupe orthogonal ou unitaire infini qui nous intéresse, dont la conjecture de Friedlander-Milnor ne dit rien, mais qui s'avère plus accessible (voir l'appendice de \cite{Mil} et les travaux de Karoubi précités). 

Nous montrons que, d'un autre côté, des groupes d'homologie dans la catégorie $\mathbf{P}(A)$ des $A$-modules à gauche projectifs de type fini ({\em bien plus accessibles} --- on dispose par exemple de nombreuses annulations gratuites en degré homologique strictement positif si $A$ est un corps de caractéristique nulle) permettent d'exprimer l'homologie à coefficients tordus polynomiaux à partir de l'homologie à coefficients constants. Plus précisément (cf. §\,\ref{par-dvg}) :

\begin{thi}\label{thigl}
 Soient $A$ un anneau muni d'une anti-involution, $F$ un foncteur polynomial de la catégorie $\mathbf{P}(A)$ vers les groupes abéliens et $n$ un entier. Il existe un isomorphisme de groupes abéliens 
$$\underset{i\in\mathbb{N}}{\col} H_n(U_{i,i}(A);F(A^{2i}))\simeq\bigoplus_{p+q=n} H_p(U_{\infty,\infty}(A);{\rm Tor}^{\mathbf{P}(A)}_q(\mathbb{Z}[T],F))$$
où $T$ est le foncteur ensembliste contravariant des formes hermitiennes (éventuellement dégénérées) et le groupe $U_{\infty,\infty}(A)=\underset{i\in\mathbb{N}}{\col} U_{i,i}(A)$ opère trivialement sur ${\rm Tor}^{\mathbf{P}(A)}_q(\mathbb{Z}[T],F)$.  
\end{thi}

 À titre d'application, nous obtenons, au corollaire~\ref{cor-azp}, le résultat suivant :

\begin{thi}\label{thie1}
 Soient $A$ un anneau commutatif sans torsion (sur $\mathbb{Z}$) où $2$ est inversible et $d\in\mathbb{N}$. Le groupe abélien gradué
$$\underset{i\in\mathbb{N}}{\col}H_*(O_{i,i}(A);(A^{2i})^{\otimes d})$$
est nul si $d$ est impair et isomorphe, pour $d=2n$, à la somme directe de $\frac{(2n)!}{2^n.n!}$ copies de
$$H_*\big(O_{\infty,\infty}(A);(A\otimes M_*)^{\otimes n}\big)$$
(avec action triviale du groupe orthogonal), où le groupe abélien gradué $M_*$ est l'homologie de Mac Lane de $\mathbb{Z}$ (qu'ont calculée Franjou et Pirashvili).
\end{thi}

Signalons également que, les groupes linéaires sur un anneau $A$ étant des groupes unitaires sur $A^{op}\times A$ (muni de l'anti-involution canonique), nos résultats s'appliquent aussi à l'homologie stable des groupes linéaires à coefficients polynomiaux. Ils nous permettent de retrouver un théorème important de Scorichenko (cf. \cite{Sco}) qui peut s'exprimer par l'isomorphisme entre $K$-théorie stable et homologie des foncteurs pour des coefficients polynomiaux --- l'équivalence entre ce point de vue et celui de l'homologie des groupes qu'on privilégie ici est traitée de manière générale dans \cite{DV}, nous n'y revenons donc pas (nous pourrions traduire nos résultats par des isomorphismes entre des groupes de {\em $K$-théorie stable hermitienne} et d'homologie des foncteurs).

 Quittons maintenant le cas des groupes unitaires hyperboliques. Pour simplifier, supposons que $A$ est un anneau commutatif contenant $1/2$ et qu'on s'intéresse aux groupes orthogonaux $O_n(A)$ associés aux formes $\sum_{i=1}^n X_i^2$ sur $A^n$. Dans certains cas (notamment celui des corps finis étudié dans \cite{DV}, où une telle forme est hyperbolique pour tout $n$ pair), le groupe orthogonal infini $O_\infty(A)$ qu'on obtient par colimite sur $n$ a la même homologie que $O_{\infty,\infty}(A)$ pour des raisons formelles ; néanmoins, en général, leur comportement homologique diffère largement. Par exemple, le problème de la stabilité homologique (même à coefficients constants) pour $O_n(A)$ s'avère plus ardu que pour $O_{n,n}(A)$ : des conditions à la Bass sur $A$ ne suffisent plus à la garantir, même lorsque $A$ est un corps --- on a besoin de conditions arithmétiques (cf. \cite{Vog} et \cite{Col}). Des difficultés similaires adviennent pour traiter d'homologie stable à coefficients polynomiaux dans un tel contexte : l'analogue du théorème précédent est par exemple en défaut, dès le degré homologique $0$, pour $A=\mathbb{Q}[X]$ (cf. §\,\ref{par-h0}). On peut toutefois obtenir des résultats partiels, par exemple :

\begin{thi}\label{tq}
 Soient $A$ un sous-anneau de $\mathbb{Q}$ contenant $1/2$ et $d\in\mathbb{N}$. Le groupe abélien gradué
$$\underset{i\in\mathbb{N}}{\col}H_*(O_i(A);(A^{i})^{\otimes d})$$
est nul si $d$ est impair et isomorphe, pour $d=2n$, à la somme directe de $\frac{(2n)!}{2^n.n!}$ copies de
$$H_*\big(O_\infty(A);(A\otimes M_*)^{\otimes n}\big)$$
(avec action triviale), où $M_*$ est l'homologie de Mac Lane de $\mathbb{Z}$.
\end{thi}

Ce résultat se déduit d'un théorème tout à fait semblable au théorème~\ref{thigl}, comme le théorème~\ref{thie1} se déduit du théorème~\ref{thigl}.

Lorsque $A$ est un corps, nous pouvons nous affranchir de toute hypothèse arithmétique, obtenant par exemple :

\begin{thi}
 Soient $k$ un corps commutatif de caractéristique nulle et $n\in\mathbb{N}^*$. On a
$$\underset{i\in\mathbb{N}}{\col}H_*(O_i(k);\Lambda^n_\mathbb{Q}(k^i))=0$$
où $\Lambda^n_\mathbb{Q}$ désigne la $n$-ème puissance extérieure sur $\mathbb{Q}$.
\end{thi}

Nous citons ce résultat car, lorsque $k$ est le corps des nombres réels, c'est une forme affaiblie d'un théorème de Dupont et Sah (qui inclut également un résultat de stabilité), établi dans \cite{DS}, qui constitue une étape cruciale dans leur approche homologique de la résolution complète du troisième problème de Hilbert (caractérisation des polyèdres euclidiens équivalents par découpage par l'égalité des volumes et des invariants de Dehn). De fait, nos résultats fournissent un grand nombre de calculs stables d'homologie des {\em groupes de Lie rendus discrets} (moyennant la connaissance de leur homologie stable à coefficients constants, dans les cas qui ne se résument pas à de l'annulation), dans un cadre général indépendant des propriétés arithmétiques du corps de base.

\paragraph*{Présentation des méthodes employées}

Le principe général des démonstrations est le même que celui suivi dans \cite{DV} : on commence par relier l'homologie stable des groupes unitaires (hyperboliques ou non) à coefficients tordus par un foncteur $F$ à l'homologie de la catégorie des espaces hermitiens {\em non dégénérés} correspondants à coefficients dans $F$. Cela se fait de façon assez formelle, sans nécessiter d'hypothèse polynomiale sur $F$. Le point plus délicat consiste à montrer que l'inclusion de la catégorie des espaces hermitiens non dégénérés dans la catégorie de tous les espaces hermitiens induit un isomorphisme en homologie à coefficients polynomiaux. Ceci s'obtient en établissant que la deuxième page de la suite spectrale de Grothendieck associée à l'extension de Kan déduite de cette inclusion est nulle hors de la première colonne. Le procédé utilisé dans \cite{DV}, reposant sur l'analyse explicite d'une catégorie de fractions très difficile à généraliser sur un anneau quelconque, a été modifié et simplifié, à l'aide de résultats d'annulation homologique dans les catégories de foncteurs remarquablement simples et efficaces développés par Scorichenko (cf. \cite{Sco}).

Dans le cas des groupes unitaires de type hyperbolique, l'application du critère de Scorichenko s'obtient sans trop de peine. Sur des corps, on peut raisonner de façon directe pour traiter de groupes de type $O_n$ (plutôt que $O_{n,n}$). Les résultats partiels que nous obtenons sur certains anneaux (des localisés d'anneaux d'entiers de corps de nombres) dans le cas non hyperbolique nécessitent un peu plus de travail : on a besoin de plusieurs lemmes établissant {\em stablement} l'existence de morphismes quadratiques vérifiant certaines propriétés lorsque les obstructions évidentes à leur existence une fois les scalaires étendus au corps des fractions sont levées.

\paragraph*{Organisation de l'article} La section~\ref{sec1} rappelle, avec quelques variations, le cadre formel de \cite{DV} pour relier homologie des foncteurs et homologie stable de groupes à coefficients tordus.

La section~\ref{secan} expose les résultats d'annulation de Scorichenko en homologie des foncteurs et met en place de manière abstraite les façons dont nous les appliquerons.

La section~\ref{sec-mon} est également préliminaire : on y établit un résultat de comparaison homologique général mettant en jeu un foncteur en monoïdes abéliens et sa symétrisation. Le résultat semble nouveau, mais repose sur des techniques très classiques en homologie des foncteurs (utilisation de suites spectrales tirées de la construction barre).

Dans la section~\ref{scg}, on présente le cadre dans lequel on va utiliser les outils précédents : les objets hermitiens dans le contexte général d'une petite catégorie additive munie d'un foncteur de dualité ; on y donne un premier résultat de comparaison entre homologie stable des groupes unitaires (dans ce cadre) et homologie des foncteurs.

La section~\ref{seccen}
démontre les résultats annoncés plus haut pour les groupes de type hyperbolique\,\footnote{On parle aussi de situation {\em générique} car tout espace hermitien non dégénéré se plonge dans un espace hermitien hyperbolique.}, y compris l'application aux résultats de Scorichenko sur l'homologie stable des groupes linéaires et quelques exemples de calculs.

La section~\ref{sng} présente nos incursions dans un cadre plus général (non nécessairement hyperbolique) : nous obtenons des résultats complets pour les catégories semi-simples, quasi complets pour le degré homologique nul, ainsi qu'un théorème plus difficile à obtenir (dont le théorème~\ref{tq} se déduit) mais valable dans un cadre arithmétique assez restreint.

L'appendice rappelle quelques définitions, notations et propriétés élémentaires de l'homologie des foncteurs, des extensions de Kan et de l'homologie de Hochschild dans ce contexte. On y énonce également quelques résultats utiles pour mener des calculs explicites.

\paragraph*{Remerciements}

L'auteur est reconnaissant à Jean-Louis Cathelineau et Gaël Collinet d'avoir attiré son attention  sur le troisième problème de Hilbert et l'homologie des groupes de Lie rendus discrets. Il a aussi bénéficié d'échanges avec G. Collinet pour comprendre des phénomènes arithmétiques qui interviennent dans l'étude de certains groupes orthogonaux.

Il remercie Vincent Franjou de lui avoir procuré les travaux non publiés de Scorichenko ainsi que pour des conversations fructueuses sur l'homologie des foncteurs.

Il remercie également Christine Vespa et Antoine Touzé pour plusieurs discussions utiles à la réalisation de ce travail.

L'auteur a bénéficié du contrat ANR BLAN08-2\_338236 (HGRT : {\em nouveaux liens entre la théorie de l'homotopie et la théorie des groupes et des représentations}). Il ne soutient pas pour autant l'ANR, dont il revendique le transfert des moyens aux laboratoires sous forme de crédits récurrents.

\paragraph*{Quelques notations et conventions}
\begin{enumerate}
 \item On note $\mathbb{Z}[-]$ le foncteur adjoint à gauche au foncteur d'oubli des groupes abéliens vers les ensembles --- autrement dit, $\mathbb{Z}[E]=\mathbb{Z}^{\oplus E}$.
\item Tous les produits tensoriels de base non spécifiée sont pris sur $\mathbb{Z}$.
\item Si $A$ est un anneau commutatif et $n$ un entier naturel, on note $S^n_A$, $T^n_A$ et $\Lambda^n_A$ respectivement les endofoncteurs $n$-ème puissance symétrique, tensorielle et extérieure des $A$-modules à gauche. L'omission de l'indice $A$ indique le choix par défaut de l'anneau des entiers.
\item Si $\C$ est une catégorie, on note ${\rm Ob}\,\C$ sa classe d'objets, $\C(a,b)$ l'ensemble des morphismes de source $a$ et de but $b$ dans $\C$ et $\C^{op}$ la catégorie opposée à $\C$.
\item Si $\C$ est une petite catégorie, on note $\C-\mathbf{Mod}$ la catégorie des foncteurs de $\C$ vers la catégorie
$\mathbf{Ab}$ des groupes abéliens et l'on pose $\mathbf{Mod}-\C=\C^{op}-\mathbf{Mod}$.
\item Si $F : \A\to\B$ est un foncteur entre petites catégories, on notera $F^* : \B-\mathbf{Mod}\to\A-\mathbf{Mod}$ le foncteur de précomposition par $F$ ; par abus, on notera de la même façon la précomposition $\mathbf{Mod}-\B\to\mathbf{Mod}-\A$ par le foncteur $\A^{op}\to\B^{op}$ correspondant canoniquement à $F$. 

D'une manière générale, on désignera souvent par une étoile en haut (resp. en bas) l'effet sur un morphisme d'un foncteur contravariant (resp. covariant) clair dans le contexte.
\item\label{nice} Si $\C$ est une catégorie et $T : \C\to\mathbf{Ens}$ un foncteur vers la catégorie des ensembles, on note $\C_T$ la catégorie dont les objets sont les couples $(c,x)$ formés d'un objet $c$ de $\C$ et d'un élément $x$ de $T(c)$, les morphismes $(c,x)\to (d,y)$ étant les morphismes $f : c\to d$ de $\C$ tels que $T(f)(x)=y$ ; on note $\pi_T : \C_T\to\C\quad (c,x)\mapsto c$ le foncteur canonique.

 Si $U$ est un foncteur $\C^{op}\to\mathbf{Ens}$, on note $\C^U=((\C^{op})_U)^{op}$ et l'on désigne par $\pi^U : \C^U\to\C$ le foncteur correspondant à $\pi_U$ entre les catégories opposées.
\item\label{nism} Si $\A$ est une catégorie, on note $\mathbf{M}(\A)$ la sous-catégorie des monomorphismes scindés de $\A$ et $\mathbf{S}(\A)$ la catégorie ayant les mêmes objets que $\A$ et dont les morphismes sont donnés par
$$\mathbf{S}(\A)(a,b)=\{(u,v)\in\A(b,a)\times\A(a,b)\,|\,u\circ v=Id_a\}$$
de sorte que $\mathbf{M}(\A)$ est l'image du foncteur $\mathbf{S}(\A)\to\A$ égal à l'identité sur les objets et associant $v$ à un morphisme $(u,v)$.
\end{enumerate}

\begin{flushleft}
D'autres notations générales sont introduites en appendice.
\end{flushleft}

\section{Un cadre pour l'homologie stable}\label{sec1}

(Ce cadre est essentiellement celui de \cite{DV}, à ceci près qu'on ne se fixe plus de sous-suite cofinale, ce qui motive l'introduction de la notion de tranche ci-après.)

\medskip

Soit $(\C,\oplus,0)$ une petite catégorie monoïdale symétrique dont l'unité $0$ est un objet initial. On dispose
donc, pour tous objets $a$ et $b$ de $\C$, d'un morphisme canonique $a\simeq a\oplus 0\xrightarrow{a\oplus (0\to b)}a\oplus b$
; toute flèche notée $a\to a\oplus b$ sans autre indication désignera par défaut ce morphisme. On fait également l'hypothèse suivante : il existe un foncteur
$${\rm Aut}_\C : \C\to\mathbf{Grp}$$
qui envoie tout objet de $\C$ sur son groupe d'automorphismes, tout isomorphisme de $\C$ sur la conjugaison qu'il
définit, et toute flèche canonique $a\to a\oplus b$ vers le morphisme ${\rm Aut}_\C(a)\to {\rm Aut}_\C(a\oplus b)$
induit par le foncteur $-\oplus b$. On demande enfin à ce foncteur que tout $f\in\C(a,b)$ soit ${\rm Aut}_\C(a)$-équivariant,
où $b$ est muni de l'action déduite de $f_* : {\rm Aut}_\C(a)\to{\rm Aut}_\C(b)$, i.e. que, pour tout $\varphi\in {\rm Aut}_\C(a)$, le diagramme
$$\xymatrix{a\ar[r]^f\ar[d]_\varphi & b\ar[d]^{f_*(\varphi)}\\
a\ar[r]^f & b
}$$
commute.

On suppose aussi vérifiée l'hypothèse suivante :

(transitivité stable de l'action des groupes d'automorphismes) {\em pour tout morphisme $f : a\to b$ de $\C$, il existe un automorphisme $\varphi$ de $a\oplus b$ faisant commuter le diagramme suivant.}
$$\xymatrix{a\ar[r]^f\ar[rrd] & b\ar[r] & a\oplus b\\
& & a\oplus b\ar[u]_\varphi
}$$

Si $F$ est un objet de $\C-\mathbf{Mod}$, on définit
$$H^{st}_*(\C;F):=\underset{c\in\C}{\col} H_*({\rm Aut}_\C(c);F(c))$$
(l'hypothèse d'équivariance précédente permet de définir les flèches). On peut remplacer $\C$ par des catégories plus petites, en particulier remplacer la colimite par une colimite {\em filtrante}. Pour cela, introduisons la définition suivante : 
\begin{defi} On dit qu'une sous-catégorie $\C'$ de $\C$ en est une {\em tranche} si elle vérifie les conditions suivantes.
\begin{enumerate}
 \item Pour tout objet $c$ de $\C$, il existe un morphisme (dans $\C$) de source $c$ et de but dans $\C'$.
 \item La catégorie $\C'$ est un ensemble ordonné cofinal à droite, c'est-à-dire qu'elle est squelettique, qu'il existe au plus un morphisme entre deux objets de $\C'$ et qu'étant donnés deux objets $a$ et $b$ de $\C'$ existe $c\in {\rm Ob}\,\C'$ tel que $\C'(a,c)$ et $\C'(b,c)$ soient non vides.
\end{enumerate}
\end{defi}

On observe que :
\begin{enumerate}
 \item la catégorie $\C$ possède toujours des tranches (par le lemme de Zorn) ;
 \item si $\C'$ est une tranche et $c$ un objet de $\C$, on dispose d'un morphisme de groupes 
 $${\rm Aut}_\C(c)\to\underset{x\in\C'}{\col}{\rm Aut}_\C(x)$$
 bien défini à conjugaison près (utiliser l'hypothèse de transitivité stable) ;
 \item par conséquent, pour toute tranche $\C'$ et tout $F\in {\rm Ob}\,\C-\mathbf{Mod}$, on dispose d'un isomorphisme canonique
 $$H^{st}_*(\C;F)\xrightarrow{\simeq}\underset{x\in\C'}{\col}H_*({\rm Aut}_\C(x);F(x))$$
 dont le but s'identifie à
 $$H_*(\underset{x\in\C'}{\col}{\rm Aut}_\C(x);\underset{x\in\C'}{\col}F(x))$$
 en vertu du caractère filtrant de $\C'$.
\end{enumerate}
(Ces observations permettent d'établir le lien avec le point de vue de \cite{DV}, qui revient essentiellement à choisir une tranche dénombrable à $\C$.)

Supposons qu'on ait fait le choix d'une tranche $\C'$. Posons
$${\rm G}_\infty:=\underset{x\in\C'}{\col}{\rm Aut}_\C(x)$$
et
$$F_\infty:=\underset{x\in\C'}{\col}F(x)$$
(qui ne dépendent, à isomorphisme près, pas de $\C'$, mais l'isomorphisme n'est pas canonique).

On considère par ailleurs les hypothèses suivantes :

\begin{defi} On dit que la catégorie monoïdale $\C$ vérifie l'hypothèse :
\begin{itemize}
 \item  (H faible) si pour tous objets $a$ et $i$ de $\C$, le morphisme canonique ${\rm Aut}_\C(a)\to {\rm Aut}_\C(a\oplus i)$ induit un isomorphisme de ${\rm Aut}_\C(a)$ sur le stabilisateur de $i\to a\oplus i$.
\item
(H forte) si pour tous objets $a$, $b$, $i$ de $\C$, l'application canonique $\C(a,b)\to\C(a\oplus i,b\oplus i)$ induit une bijection de $\C(a,b)$ sur l'ensemble des morphismes rendant commutatif le diagramme
$$\xymatrix{i\ar[r]\ar[rd] & a\oplus i\ar[d]\\
& b\oplus i
}$$
\end{itemize}
\end{defi}

\begin{rem}\label{rq-sthyp}
 Toutes les hypothèses introduites dans cette section sont stables par sous-catégorie pleine monoïdale de $\C$.
\end{rem}

\begin{pr}[Cf. \cite{DV}, §\,2.1]\label{pr-dvks}
Sous l'hypothèse (H faible), il existe un isomorphisme gradué canonique 
 $$H^{st}_*(\C;F)\simeq H_*(\C\times {\rm G}_\infty;F)$$
 où l'on note par abus encore $F$ la composée de ce foncteur avec le foncteur de projection $\C\times {\rm G}_\infty\to\C$.
\end{pr}

(Pour éviter le recours à une tranche, on peut écrire le membre de droite sous la forme $\underset{c\in\C}{\col} H_*(\C\times {\rm Aut}_\C(c);F)$)

\begin{proof}(On suit à peu de choses près \cite{DV} dans un cadre plutôt plus simple.)
 
Considérons, pour tout objet $c$ de $\C$, le foncteur ${\rm Aut}_\C(c)\to\C\times {\rm Aut}_\C(c)$ dont les
composantes sont l'inclusion d'image $c$ et l'identité : il induit un morphisme de foncteurs homologiques
$H_*({\rm Aut}_\C(c);F(c))\to H_*(\C\times {\rm Aut}_\C(c);F)$. Par passage à la colimite, on en déduit un
morphisme de foncteurs homologiques $H^{st}_*(\C;F)\to H_*(\C\times {\rm G}_\infty;F)$ ; il suffit de vérifier
que c'est un isomorphisme lorsque $F$ est un objet projectif
$P^\C_a=\mathbb{Z}[\C(a,-)]$ de $\C-\mathbf{Mod}$.

Le groupe gradué $H_*(\C\times {\rm Aut}_\C(c);P^\C_a)$ s'identifie canoniquement à $H_*({\rm Aut}_\C(c))$ (grâce à la proposition~\ref{pr-bifo}), tandis qu'on a
 $$H_*({\rm Aut}_\C(c);\mathbb{Z}[\C(a,c)])\simeq\bigoplus H_*(Stab)$$
 où la somme est prise sur les orbites du ${\rm Aut}_\C(c)$-ensemble $\C(a,c)$, tandis que $Stab$ désigne les stabilisateurs correspondants. Via ces isomorphismes, notre morphisme est induit par les inclusions de ces stabilisateurs dans ${\rm Aut}_\C(c)$. La conclusion résulte alors des deux faits suivants :
 \begin{enumerate}
  \item si $c$ est de la forme $a\oplus i$, le stabilisateur correspondant au morphisme canonique $a\to a\oplus i$ s'identifie à ${\rm Aut}_\C(i)$ (hypothèse (H faible)), et les morphismes $H_*({\rm Aut}_\C(i))\to H_*({\rm Aut}_\C(a\oplus i))$ induisent par passage à la colimite sur $i$ un isomorphisme ;
  \item l'hypothèse de transitivité stable montre que l'image de toute orbite de $\C(a,c)$ dans $\C(a,a\oplus c)$ est celle du morphisme canonique. 
 \end{enumerate}
\end{proof}

On dispose donc d'un lien étroit entre $H^{st}_*(\C;F)$, l'homologie à coefficients triviaux de
${\rm G}_\infty$ et $H_*(\C;F)$ :
\begin{cor}\label{cor-dvks}
 Sous les mêmes hypothèses, on dispose d'isomorphismes naturels
$$H^{st}_n(\C;F)\simeq\bigoplus_{p+q=n}{\rm Tor}^\C_p(H_q({\rm G}_\infty;\mathbb{Z}),F)$$
(où le membre de gauche du groupe de torsion est un foncteur constant)
et d'isomorphismes (non nécessairement naturels)
$$H^{st}_n(\C;F)\simeq\bigoplus_{p+q=n}H_p({\rm G}_\infty;H_q(\C;F))$$
(où ${\rm G}_\infty$ opère trivialement).
\end{cor}
(Cela s'obtient à partir de la proposition~\ref{pr-dvks} par des arguments usuels d'algèbre homologique, voir par exemple \cite{DV}, §\,2.3 pour des détails). On peut ainsi penser à $H_*(\C;F)$
comme à une {\em $K$-théorie stable} associée aux coefficients définis par $F$, dans le contexte de la catégorie $\C$ (cf. ibidem).

\begin{rem}\label{rq-cadg}
 Le contexte général décrit dans cette section est susceptible d'un certain nombre de généralisations :
\begin{enumerate}
 \item on peut remplacer les groupes d'automorphismes de la catégorie $\C$ par des groupes munis d'un morphisme naturel vers iceux,
avec des hypothèses adaptées. Cela peut permettre, par exemple, dans les considérations suivantes, de remplacer
les groupes linéaires ou orthogonaux par leurs analogues spéciaux (dans le cas d'un anneau de base commutatif) ;
\item l'hypothèse de transitivité stable est donnée sous une forme plus forte que ce dont on a réellement besoin
(cf. \cite{DV} pour la condition minimale utile) ; la restriction adoptée ici est motivée par sa simplicité et
sa conservation par toute sous-catégorie monoïdale pleine, qui la rendent très naturelle ;
\item il n'est pas indispensable de travailler dans une catégorie monoïdale symétrique (on peut notamment souhaiter
utiliser certaines sous-catégories remarquables non monoïdales d'une telle catégorie --- par exemple pour traiter
d'homologie stable de groupes orthogonaux du type $O_{n,i}$, où $i$ est fixé). Néanmoins, nous ne sommes pas
parvenus à trouver des axiomes pleinement satisfaisants pour traiter de telles généralisations.
\end{enumerate}
\end{rem}

\section{Outils d'annulation en homologie des foncteurs}\label{secan}

Dans les deux premiers paragraphes de cette section, on présente des résultats d'annulation en homologie des foncteurs remarquablement simples, généraux et efficaces, dus à Scorichenko \cite{Sco}. Bien que nous ne fassions que suivre ses arguments, nous les reproduisons dans leur intégralité afin que le lecteur dispose d'une présentation complète, avec nos notations, de ce travail resté non publié.

Dans le dernier paragraphe de la présente section, nous mettons en place un cadre général dans lequel s'inséreront nos applications des résultats de Scorichenko.

\subsection{Version abstraite des résultats d'annulation de Scorichenko}\label{sec-scoabs}

Dans cette section, on se donne des catégories abéliennes $\E_1$, $\E_2$, $\E_3$ et un bifoncteur semihomologique à droite $\E_1\times\E_2\to\E_3$ concentré en degrés positifs $(\Hh_n)_{n\in\mathbb{N}}$, c'est-à-dire une suite de foncteurs $\Hh_n : \E_1\times\E_2\to\E_3$ telle que, pour tout $A\in {\rm Ob}\,\E_1$, $\Hh_*(A,-)$ soit additif et muni d'une structure de foncteur homologique (morphisme de liaison donnant lieu à une suite exacte longue naturelle pour toute suite exacte courte) --- on pourrait imposer une condition de naturalité en $A$, mais ce ne sera même pas nécessaire. On suppose également que, pour tout $B\in {\rm Ob}\,\E_2$, $\Hh_*(0,B)=0$.

On se donne aussi des endofoncteurs {\em exacts} 
 $\Phi$ et $\Psi$ de $\E_1$ et $\E_2$ respectivement, munis de transformations naturelles $\sigma : \Phi\to Id_{\E_1}$ et $\tau : \Psi\to Id_{\E_2}$ de sorte qu'existe un isomorphisme naturel gradué $\Hh_*(\Phi A,B)\simeq\Hh_*(A,\Psi B)$ rendant commutatif le diagramme
$$\xymatrix{\Hh_*(\Phi A,B)\ar[d]^-\simeq\ar[rr]^-{\Hh(\sigma_A,B)} & & \Hh_*(A,B)\\
\Hh_*(A,\Psi B)\ar[urr]_-{\Hh(A,\tau_B)} & & 
}$$

On note $\Pp$ la classe des objets $A$ de $\E_1$ tels que $\sigma_A=0$.

Notons $\Theta$ le noyau de $\tau$. On définit des classes d'objets $\s_d$ de $\E_2$ comme suit :
\begin{enumerate}
 \item $\s_0$ est la classe des $B$ tels que $\tau_B$ soit un épimorphisme ;
 \item pour $d\in\mathbb{N}^*$, $\s_d$ est la classe des $B$ tels que $\Theta^i B$ ($i$-ème itération de $\Theta$ sur $B$) appartienne à $\s_0$ pour $0\leq i\leq d$.
\end{enumerate}

\begin{pr}\label{sco-abs1}
 Si $A$ appartient à $\Pp$ et $B$ à $\s_d$, alors $\Hh_n(A,B)=0$ pour $n\leq d$.
\end{pr}

\begin{proof}
On procède par récurrence sur $d$.
\begin{itemize}
 \item Pour $d=0$ : $\Hh_0(A,\tau_B)$ est un \'epimorphisme comme $\tau_B$ ($\Hh_0(A,-)$ est exact à droite), mais il est aussi nul puisqu'il s'identifie à $\Hh_0(\sigma_A,B)$ et que $\sigma_A$ égale $0$ (on rappelle que $\Hh_*(0,B)=0$). Donc son but $\Hh_0(A,B)$ est nul.
 \item Si l'assertion est vraie pour $\s_{d-1}$ et que $B$ est dans $\s_{d}$, la suite exacte courte $0\to\Theta B\to\Psi B\xrightarrow{\tau_B} B\to 0$ procure une suite exacte
 $$\Hh_{d}(A,\Psi B)\to\Hh_{d}(A,B)\to\Hh_{d-1}(A,\Theta B)\;;$$
 $\Hh_{d-1}(A,\Theta B)$ est nul par l'hypoth\`ese de r\'ecurrence, et la fl\`eche $\Hh_{d}(A,\Psi B)\to\Hh_{d}(A,B)$ induite par $\tau_B$ est nulle comme $\sigma_A$ (cf. cas précédent), d'o\`u la nullit\'e de $\Hh_{d}(A,B)$.
\end{itemize}
\end{proof}

On cherche maintenant, suivant Scorichenko, des critères plus maniables pour vérifier qu'un objet de $\E_2$ appartient à tous les $\s_d$. On introduit pour cela une première hypothèse supplémentaire :
\begin{hyp}\label{hyp-sco1}
 Il existe un endomorphisme $\gamma$ du foncteur $\Psi^2$ tel que $\Psi\tau\circ\gamma=\tau_\Psi$ et $\tau_\Psi\circ\gamma=\Psi\tau$.
\end{hyp}

\begin{pr}\label{sco-abs2}
 Supposons que l'hypothèse~\ref{hyp-sco1} est vérifiée. Si $A$ est un objet de $\Pp
$ et $B$ un objet de $\E_2$ tel que $\tau_B$ soit un épimorphisme {\em scindé}, alors $\Hh_*(A,B)=0$.
\end{pr}

On s'abstiendra de donner la démonstration (aisée) de cet énoncé, car le deuxième critère que nous énonçons maintenant (toujours suivant Scorichenko) est plus général --- en effet, la condition que $\tau_B$ soit un épimorphisme scindé s'avère souvent trop restrictive.

Pour cela, on suppose donnée une catégorie abélienne $\E_4$ et un foncteur {\em exact et fidèle} $R : \E_2\to\E_4$. On note $\s co$ la classe des objets $B$ de $\E_2$ tels que $R\tau_B : R\Psi B\to RB$ soit un épimorphisme scindé de $\E_4$.

\begin{hyp}\label{hyp-sco2}
 Il existe un endofoncteur exact $\bar{\Psi}$ de $\E_4$ et une transformation naturelle $\bar{\tau} : \bar{\Psi}\to Id_{\E_4}$ munis d'un isomorphisme $R\Psi\simeq\bar{\Psi}R$ faisant commuter le diagramme
 $$\xymatrix{R\Psi\ar[r]^-{R\tau}\ar[d]_-\simeq  & R \\
 \bar{\Psi} R\ar[ru]_-{\bar{\tau}_R} &
 }$$
\end{hyp}

\begin{pr}[Scorichenko]\label{sco-abs3}
 Supposons les hypothèses~\ref{hyp-sco1} et~\ref{hyp-sco2} satisfaites. Si $A$ est un objet de $\Pp$ et $B$ un objet de $\s co$, alors $\Hh_*(A,B)=0$.
\end{pr}

\begin{proof}
La fidélité et l'exactitude de $R$ impliquent que, si $R\tau_B$ est un épimorphisme, alors $\tau_B$ en est un également. Compte tenu de la proposition~\ref{sco-abs1}, il suffit donc d'établir la stabilité de $\s co$ par le foncteur $\Theta$.

Soient $B$ un objet de $\s co$ et $s$ une section de $R\tau_B$. Notons $s'$ le morphisme :
 $$R\Psi B\simeq\bar{\Psi} RB\xrightarrow{\bar{\Psi}s}\bar{\Psi}R\Psi B\simeq R\Psi^2 B\xrightarrow{R\gamma_B}R\Psi^2 B.$$
 Le diagramme commutatif
 $$\xymatrix{\bar{\Psi} RB \ar[r]^-{\bar{\Psi}s}\ar[rd]_-{Id} & \bar{\Psi}R\Psi B\ar[r]^-\simeq\ar[d]^{\bar{\Psi}R\bar{\tau}_B} & R\Psi^2 B\ar[r]^-{R\gamma_B}\ar[d]_{R\Psi\tau_B} & R\Psi^2 B\ar[dl]^-{R\tau_{\Psi B}}\\
 & \bar{\Psi}RB\ar[r]^-\simeq & R\Psi B &
 }$$
 montre que $s'$ est une section de $R\tau_{\Psi B}$.
 
 De m\^eme, la commutation du diagramme
 $$\xymatrix{R\Psi B\ar[r]^\simeq\ar[rd]_{R\tau_B} & \bar{\Psi}RB\ar[d]^{\bar{\tau}_{RB}}\ar[r]^{\bar{\Psi}s} & \bar{\Psi}R\Psi B\ar[rd]_{\bar{\tau}_{R\Psi B}}\ar[r]^\simeq & R\Psi^2 B\ar[d]^{R\tau_{\Psi B}}\ar[r]^{R\gamma_B} & R\Psi^2 B\ar[dl]^{R\Psi\tau_B} \\
 & RB\ar[rr]^s & &R\Psi B & 
 }$$
 montre que les sections $s$ et $s'$ sont compatibles au morphisme $\tau_B : \Psi B\to B$ au sens où le diagramme
 $$\xymatrix{R\Psi B\ar[r]^-{s'}\ar[d]_-{R\tau_B} & R\Psi^2 B\ar[d]^-{R\Psi\tau_B}\\
 RB\ar[r]^-s & R\Psi B
 }$$
 commute, de sorte qu'elles d\'efinissent par passage aux noyaux une section de $R\tau_{\Theta B}$.
\end{proof}

\subsection{Application aux foncteurs polynomiaux}

Soit $\A$ une petite catégorie additive.

Si $E$ est un ensemble fini non vide, on note $t_E : \A\to\A$ le foncteur $A\mapsto A^{\oplus E}$. Ce foncteur
est adjoint à lui-même, les foncteurs de précomposition $T_E$ qu'il induit $\A-\mathbf{Mod}\to\A-\mathbf{Mod}$
et $\mathbf{Mod}-\A\to\mathbf{Mod}-\A$ (qu'on notera par abus de la même façon) sont donc munis d'un
isomorphisme naturel
\begin{equation}\label{eqadjc}
 {\rm Tor}^\A_*(T_E(X),F)\simeq {\rm Tor}^\A_*(X,T_E(F)).
\end{equation}

Si $I$ est une partie de $E$, on note $u_I : Id\to t_E$ (la mention de $E$ est omise par abus) la
transformation naturelle dont la composante $Id\to Id$ associée à $e\in E$ est l'identité si $e\in I$ et $0$
sinon, et $p_I : t_E\to Id$ la transformation naturelle~\guillemotleft~duale~\guillemotright~(donnée par les
mêmes composantes). On définit des transformations naturelles d'endofoncteurs de $\A-\mathbf{Mod}$
$$cr_E^{\A,dir}=\sum_{I\subset E}(-1)^{\# I}(u_I)_* : Id\to T_E$$
(effet croisé direct ; $\#$ désigne le cardinal) et
$$cr_E^{\A,inv}=\sum_{I\subset E}(-1)^{\# I}(p_I)_* : T_E\to Id$$
(effet croisé inverse). Les exposants seront omis s'il n'y a pas d'ambiguïté. 

Pour tout foncteur $F\in {\rm Ob}\,\A-\mathbf{Mod}$ et tout entier $d\in\mathbb{N}$, l'équivalence entre les
propriétés suivantes est classique :
\begin{enumerate}
 \item $cr_E^{\A,dir}(F)=0$ pour tout ensemble $E$ de cardinal au moins $d+1$ ;
\item $cr_E^{\A,dir}(F)=0$ pour un ensemble $E$ de cardinal $d+1$ ;
 \item $cr_E^{\A,inv}(F)=0$ pour tout ensemble $E$ de cardinal au moins $d+1$ ;
\item $cr_E^{\A,inv}(F)=0$ pour un ensemble $E$ de cardinal $d+1$.
\end{enumerate}
(Cf. par exemple la référence originelle d'Eilenberg-Mac Lane \cite{EML}, chap.~II, d'où l'on peut aussi tirer que cette définition équivaut à exiger que la fonction $\A(a,b)\to\mathbf{Ab}(F(a),F(b))$ qu'induit $F$ soit polynomiale de degré au plus $d$ pour tous objets $a$ et $b$ de $\A$.)

Si ces conditions sont vérifiées, on dit que $F$ est {\em polynomial de degré au plus $d$}. Un foncteur est dit
{\em analytique} s'il est colimite de foncteurs polynomiaux (on peut toujours choisir cette colimite filtrante,
ce qui permet de propager les propriétés d'annulation homologique des foncteurs polynomiaux aux foncteurs
analytiques).

Pour la commodité de nos considérations ultérieures, on s'intéressera à la variante suivante des effets croisés.
Supposons que $(E,e)$ est ensemble fini pointé. On définit comme précédemment
des transformations naturelles d'endofoncteurs de $\A-\mathbf{Mod}$
$$cr_{E,e}^{\A,dir}=\sum_{e\in I\subset E}(-1)^{\# I}(u_I)_* : Id\to T_E$$
(effet croisé direct pointé) et
$$cr_{E,e}^{\A,inv}=\sum_{e\in I\subset E}(-1)^{\# I}(p_I)_* : T_E\to Id$$
(effet croisé inverse pointé).

On vérifie facilement que la condition que $F$ est polynomial de degré au plus $d$ équivaut encore à chacune
des assertions suivantes : 
\begin{enumerate}
 \item $cr_{E,e}^{\A,dir}(F)=0$ pour tout ensemble pointé $(E,e)$ de cardinal au moins $d+2$ ;
\item $cr_{E,e}^{\A,dir}(F)=0$ pour un ensemble pointé $(E,e)$ de cardinal $d+2$ ;
 \item $cr_{E,e}^{\A,inv}(F)=0$ pour tout ensemble pointé $(E,e)$ de cardinal au moins $d+2$ ;
\item $cr_{E,e}^{\A,inv}(F)=0$ pour un ensemble pointé $(E,e)$ de cardinal $d+2$.
\end{enumerate}

Modulo
l'identification~(\ref{eqadjc}), on dispose d'isomorphismes naturels
\begin{equation}\label{eqadjcr}
 {\rm Tor}^\A_*(cr_{E,e}^{\A^{op},inv},Id)\simeq  {\rm Tor}^\A_*(Id,cr_{E,e}^{\A,inv})
\end{equation}
(de même avec les effets croisés directs et les variantes non pointées).

Dans \cite{Sco} est établi le théorème suivant (pour l'homologie des bifoncteurs plutôt que les groupes de
torsion entre deux foncteurs, et dans le cadre des effets croisés ordinaires plutôt que la variante pointée qu'on leur préfère ici pour des raisons techniques, mais la démonstration est identique). On y note
$\theta : \mathbf{M}(\A)\to\A$ le foncteur d'inclusion de la sous-catégorie des monomorphismes scindés de~$\A$.

\begin{thm}[Scorichenko]\label{th1-sco}
Soient $d\in\mathbb{N}$, $(E,e)$ un ensemble pointé de cardinal $d+2$,  $X\in {\rm Ob}\,\mathbf{Mod}-\A$ et
$F\in {\rm Ob}\,\A-\mathbf{Mod}$. On suppose que :
\begin{itemize}
 \item $F$ est polynomial de degré au plus $d$ ;
 \item le morphisme $\theta^* cr_{E,e}(X) : \theta^*T_E(X)\to\theta^* X$ est un épimorphisme scindé
de $\mathbf{Mod}-\mathbf{M}(\A)$.
\end{itemize}

Alors ${\rm Tor}^\A_*(X,F)=0$.
\end{thm}

\begin{proof}
 On applique le formalisme de la section~\ref{sec-scoabs} avec les choix suivants : $\E_1=\A-\mathbf{Mod}$,
$\E_2=\mathbf{Mod}-\A$, $\E_3=\mathbf{Ab}$, $\Hh_*(A,B)={\rm Tor}^\A_*(B,A)$, $\Phi=T_E$, $\Psi=T_E$,
$\sigma=cr_{E,e}^{\A,inv}$, $\tau=cr_{E,e}^{\A^{op},inv}$ : les observations précédentes montrent que les
conditions requises sont satisfaites. L'hypothèse~\ref{hyp-sco1} est vérifiée en prenant pour $\gamma$
l'involution de $(T_E)^2\simeq T_{E\times E}$ intervertissant les deux facteurs du produit cartésien.
 
On pose enfin $\E_4=\mathbf{Mod}-\mathbf{M}(\A)$ et l'on prend pour $R$ le foncteur $\theta^*$, qui est fidèle
puisque $\theta$ est essentiellement surjectif. Le foncteur $\bar{\Psi}$ est donné comme $T_E$ par la
précomposition par $A\mapsto A^{\oplus E}$ (qui induit bien un endofoncteur de $\mathbf{M}(\A)$) et pour
$\bar{\tau}$ la transformation naturelle donnée par la même formule que $cr_{E,e}^{\A^{op},inv}$, ce qui fait
sens puisque les morphismes $p_I^{\A^{op}}$, correspondant aux morphismes $u_I$ de $\A$, sont bien dans
$\mathbf{M}(\A)$ lorsque $I$ est une partie non vide de $E$.
 
 Le résultat est donc un cas particulier de la proposition~\ref{sco-abs3}.
\end{proof}

\begin{cor}\label{corsco}
 Soit $X\in {\rm Ob}\,\mathbf{Mod}-\A$ tel que le morphisme
$\theta^* cr_{E,e}(X) : \theta^*T_E(X)\to\theta^* X$ est un épimorphisme scindé
pour tout ensemble fini pointé $(E,e)$. Alors ${\rm Tor}^\A_*(X,F)=0$ si $F\in {\rm Ob}\,\A-\mathbf{Mod}$ est
analytique.
\end{cor}

Ces assertions constituent un raffinement du résultat suivant, qui est immédiat (cas particulier de la proposition~\ref{sco-abs1} en degré homologique nul) mais utile :
\begin{pr}\label{prevpol}
 \begin{enumerate}
  \item Supposons que $(E,e)$ est un ensemble fini pointé de cardinal $d+2$ et $G$ un objet de $\mathbf{Mod}-\A$ tel que l'effet croisé pointé $cr_{E,e} : T_E(G)\to G$ soit surjectif. Alors $G\underset{\A}{\otimes} F=0$ pour tout $F\in {\rm Ob}\,\A-\mathbf{Mod}$ polynomial de degré au plus~$d$.
\item Si la surjectivité de $cr_{E,e} : T_E(G)\to G$ a lieu pour tous les ensembles finis pointés $(E,e)$, alors $G\underset{\A}{\otimes} F=0$ pour tout $F\in {\rm Ob}\,\A-\mathbf{Mod}$ analytique.
 \end{enumerate}
\end{pr}

\subsection{Utilisation dans certaines catégories monoïdales}\label{s-mos}

On commence par une propriété élémentaire très générale des extensions de Kan dérivées dans un contexte monoïdal :

\begin{pr}\label{pr-ausco1}
 Soient $n\in\mathbb{N}$, $(\C,\oplus,0)$ une petite catégorie monoïdale symétrique dont l'unité est objet initial et $\D$ une sous-catégorie monoïdale pleine de $\C$. Notons $\varphi : \D\to\C$ le foncteur d'inclusion et $X=\mathcal{L}^\varphi_n(\mathbb{Z})\in\mathbf{Mod}-\C$ (où la notation $\mathcal{L}^\varphi_*$ est définie par~(\ref{eq-kum}) dans l'appendice). Alors :
\begin{enumerate}
 \item $X$ est nul sur $\D$ ;
 \item il existe un morphisme $X(c)\to X(c\oplus d)$ naturel en les objets $c$ de $\C$ et $d$ de $\D$ qui est une section du morphisme $X(c\oplus d)\to X(c)$ induit par la flèche canonique $c\simeq c\oplus 0\to c\oplus d$.
\end{enumerate}
\end{pr}

\begin{proof}
 On rappelle (cf. remarque~\ref{rq-ekz}) que $X$ est donné par
$$X(c)=\tilde{H}_*(\D_{\varphi^* \C(c,-)}).$$

Si $c$ appartient à $\D$, $(c,Id_c)$ est objet initial de $\D_{\varphi^*\C(c,-)}$, donc $X(c)$ est nul.

Comme $\D$ est monoïdale, pour $c\in {\rm Ob}\,\C$ et $d\in {\rm Ob}\,\D$, l'endofoncteur $-\oplus d$ de $\C$ induit un foncteur
$$\D_{\varphi^* \C(c,-)}\to\D_{\varphi^* \C(c\oplus d,-)}\qquad (x,f)\mapsto (x\oplus d,f\oplus Id_d)$$
qui induit un morphisme $X(c)\to X(c\oplus d)$ naturel en $c$ et $d$\,\footnote{La naturalité en $c$ est évidente (vraie au niveau des foncteurs) ; celle en $d$ s'obtient grâce à une transformation naturelle, par un argument identique à celui utilisé pour terminer la démonstration.}. La post-composition de ce morphisme avec le morphisme $X(c\oplus d)\to X(c)$ induit par $c\to c\oplus d$ est induite par le foncteur
$$\D_{\varphi^* \C(c,-)}\to\D_{\varphi^* \C(c,-)}\qquad (x,f)\mapsto (x\oplus d,c\to c\oplus d\xrightarrow{f\oplus Id_d}x\oplus d).$$

Il existe une transformation naturelle de l'identité de $\D_{\varphi^* \C(c,-)}$ vers ce foncteur, induite par la flèche canonique $x\to x\oplus d$, puisque le diagramme
$$\xymatrix{c\ar[rr]^f\ar[d] & & x\ar[d]\\
c\oplus d\ar[rr]^-{f\oplus Id_d} & & x\oplus d
}$$
commute. Ainsi, notre composée égale l'identité, d'où la proposition.
\end{proof}

Dans la suite de ce paragraphe, on se donne une petite catégorie additive $\A$ et un foncteur monoïdal $T : \A^{op}\to\mathbf{Ens}$. La structure monoïdale sur $\A$ est celle donnée par la somme directe, et sur les ensembles par le produit direct ; on entend ici monoïdal au sens faible : on a des fonctions naturelles $T(A)\times T(B)\to T(A\oplus B)$ (et un élément de $T(0)$) vérifiant les conditions de cohérence habituelles. Il revient au même de se donner une factorisation de $T$ à travers le foncteur d'oubli des monoïdes commutatifs vers les ensembles. On suppose également que $T(0)$ est réduit à un élément. Alors la somme directe induit une structure monoïdale symétrique, notée $\overset{T}{\oplus}$ ou simplement $\oplus$, sur la catégorie $\C=\A^T$ (voir \ref{nice} en fin d'introduction pour la notation) : $(A,x)\overset{T}{\oplus}(B,y)$ est $A\oplus B$ muni de l'image de $(x,y)$ par $T(A)\times T(B)\to T(A\oplus B)$, l'unité $0$ (muni de l'unique élément de $T(0)$) est objet initial de $\C$, et le foncteur canonique $\pi=\pi^T : \C\to\A$ est monoïdal au sens fort ($\pi(a\oplus b)\simeq\pi(a)\oplus\pi(b)$ avec conditions de cohérence).

On se donne aussi une sous-catégorie pleine monoïdale $\D$ de $\C$ et l'on note $\varphi : \D\to\C$ l'inclusion.

\begin{pr}\label{pr-auxh0}
 Supposons que la propriété suivante est vérifiée :

{\rm pour tout objet $(A,x)$ de $\C$, il existe  $y\in T(A)$ et $d\in {\rm Ob}\,\D$ tel que $(A,x+y)$ appartienne à $\D$ et que $\C((A,y),d)$ soit non vide.}

Alors, pour tout foncteur $X\in {\rm Ob}\,\mathbf{Mod}-\C$ satisfaisant aux deux conditions de la proposition~\ref{pr-ausco1} et tout foncteur analytique $F\in {\rm Ob}\,\A-\mathbf{Mod}$, on a $X\underset{\C}{\otimes}\pi^*(F)=0$.
\end{pr}

\begin{proof}
 On rappelle qu'on dispose d'un isomorphisme naturel (cf. appendice)
$$X\underset{\C}{\otimes}\pi^*(F)\simeq\pi_!(X)\underset{\A}{\otimes}F$$
où
$$\pi_!(X)(A)=\bigoplus_{x\in T(A)}X(A,x).$$

Il suffit donc de montrer (cf. proposition~\ref{prevpol}) que, pour tout ensemble fini pointé $(E,e)$, l'effet croisé
$$cr_{E,e} : \pi_!(X)(A^{\oplus E})\to\pi_!(X)(A)$$
est surjectif.

Soit $x\in T(A)$ ; choisissons $y\in T(A)$, $d\in {\rm Ob}\,\D$ comme dans l'hypothèse et un morphisme $a : (A,y)\to d$. Examinons la composée
$$X(A,x)\to X((A,x)\oplus d^{\oplus E\setminus e})\xrightarrow{X((A,x)\oplus a^{\oplus E\setminus e})} X((A,x)\oplus (A,y)^{\oplus E\setminus e})\dots$$
$$\hookrightarrow\pi_!(X)(A^{\oplus E})\to\pi_!(X)(A)$$
où la première flèche est donnée par la première propriété de la proposition~\ref{pr-ausco1} et la dernière est l'effet croisé pointé. Elle a une composante vers $X(A,x)$ égale à l'identité (correspondant à la partie pointée $\{e\}$ de $E$), puisque la première flèche est une section du morphisme induit par $(A,x)\to (A,x)\oplus d^{\oplus E\setminus e}$. Si maintenant $I$ est une partie pointée de $E$ à $n+1$ éléments, le terme de l'effet croisé auquel contribue $I$ arrive dans le facteur $X(A,x+ny)$ de $\pi_!(X)(A)$ ; il est induit par la flèche qu'induit $u_I : (A,x+ny)\to (A,x)\oplus (A,y)^{\oplus E\setminus e}$. Mais celle-ci se factorise comme suit :
$$\xymatrix{(A,x+ny)\ar[rr]^-{u_I}\ar[rd]_-\delta & & (A,x)\oplus (A,y)^{\oplus E\setminus e} \\
& (A,x+y)\oplus (A,y)^{\oplus I\setminus\{e,i\}}\ar[ru]_-\alpha &
}$$
où $i$ est un élément arbitraire de $E\setminus e$, $\delta$ est la diagonale et $\alpha$ est donnée par la somme de la diagonale $(A,x+y)\to (A,x)\oplus (A,y)$ et de la flèche canonique $(A,y)^{\oplus I\setminus\{e,i\}}\to (A,y)^{\oplus E\setminus e}$.

Comme $(A,x+y)\oplus (A,y)^{\oplus I\setminus\{e,i\}}$ appartient à $\D$, $X$ est nul sur cet objet, de sorte que la composée qu'on étudie est l'identité. La proposition~\ref{prevpol} donne donc la conclusion.
\end{proof}

\begin{cor}\label{cor-auxh0}
 Sous les mêmes hypothèses sur $\C$ et $\D$, pour tout foncteur analytique $F\in {\rm Ob}\,\A-\mathbf{Mod}$, on dispose d'isomorphismes canoniques
$$H_0(\D;\varphi^*\pi^*F)\simeq H_0(\C;\pi^*F)\simeq\mathbb{Z}[T]\underset{\A}{\otimes} F.$$
\end{cor}

La proposition suivante généralise ce qui précède aux groupes de torsion supérieurs, sous des hypothèses supplémentaires.

\begin{pr}\label{pr-auxhg}
 Supposons qu'il existe un endofoncteur $\Phi$ de $\C$ et une sous-catégorie $\B$ de $\C$ vérifiant les propriétés suivantes :
\begin{enumerate}
 \item $\pi\circ\Phi=\pi$ ;
\item si $f$ est un morphisme de $\C$ tel que $\pi(f)$ appartienne à $\mathbf{M}(\A)$, alors $\Phi(f)$ appartient à $\B$ ;
\item pour tout objet $c$ de $\C$, la sous-catégorie $\K(c)$ de $\C_{\C(\Phi(c),-)}$ formée des morphismes de $\B$ dont le but appartient à $\D$ est connexe ;
\item de plus, il existe un objet $f : \Phi(c)\to d$ de $\K(c)$ tel que la flèche 
$$\tilde{c}\xrightarrow{g} c\oplus\Phi(c)\xrightarrow{c\oplus f}c\oplus d$$
où $g$ est la flèche donnée par la diagonale $\pi(c)\to\pi(c)\oplus\pi(c)$ (et $\tilde{c}$ est l'objet $(\pi(c),x+y)$ de $\C$, où $c=(\pi(c),x)$ et $\Phi(c)=(\pi(c),y)$) se factorise par un objet de $\D$.
\end{enumerate}

Alors pour tout foncteur $X\in {\rm Ob}\,\mathbf{Mod}-\C$ satisfaisant aux conditions de la proposition~\ref{pr-ausco1} et tout foncteur analytique $F\in {\rm Ob}\,\A-\mathbf{Mod}$, on a
$${\rm Tor}^\C_*(X,\pi^*(F))=0.$$
\end{pr}

\begin{proof}
 On va appliquer le corollaire~\ref{corsco} au foncteur $\pi_!(X)\in\mathbf{Mod}-\A$ (ce qui suffit puisque l'isomorphisme d'adjonction rappelé dans la démonstration précédente, entre $\pi_!$ et $\pi^*$, s'étend aux groupes de torsion). On se donne un ensemble pointé fini $(E,e)$.

Soient $c$ un objet de $\C$ et $f : \Phi(c)\to d$ un objet de $\K(c)$. On définit
$$s_c^f : X(c)\to X(c\oplus\Phi(c)^{\oplus E\setminus e})$$
comme la composée
$$X(c)\to X(c\oplus d^{\oplus E\setminus e})\xrightarrow{X(c\oplus f^{\oplus E\setminus e})}X(c\oplus\Phi(c)^{\oplus E\setminus e})$$
(où la première flèche est donnée par la première propriété de la proposition~\ref{pr-ausco1}).

Vérifions que $s_c^f$ ne dépend en fait pas de $f$ (de sorte qu'on notera seulement $s_c$ cette flèche). Supposons pour cela que $f$ et $f'$ sont deux objets de $\K(c)$ reliés par une flèche $g$ : on dispose en particulier d'un diagramme commutatif
$$\xymatrix{\Phi(c)\ar[r]^-f\ar[rd]_-{f'} & d\ar[d]^-g \\
& d'
}$$
dans $\C$. On en déduit un diagramme commutatif
$$\xymatrix{X(c)\ar[r]\ar[rd] & X(c\oplus d^{\oplus E\setminus e})\ar[rrr]^-{X(c\oplus f^{\oplus E\setminus e})} & & & X(c\oplus\Phi(c)^{\oplus E\setminus e}) \\
& X(c\oplus d'^{\oplus E\setminus e})\ar[u]_-{X(c\oplus g^{\oplus E\setminus e})}\ar[rrru]_-{X(c\oplus f'^{\oplus E\setminus e})} & & &
}$$
(le triangle de gauche commute en raison de la naturalité des sections de la proposition~\ref{pr-ausco1}) qui montre que $s^f_c=s^{f'}_c$, lorsqu'une flèche de $\K(c)$ relie $f$ et $f'$. L'hypothèse de connexité de la catégorie $\K(c)$ établit donc l'indépendance en $f$ voulue.

Pour tout objet $A$ de $\A$, les différents morphismes $s_{(A,x)}$ constituent les composantes d'un morphisme $\sigma_A : \pi_!(X)(A)\to\pi_!(X)(A^{\oplus E})$. Montrons que cette collection de morphismes définit une transformation naturelle $\theta^*\pi_!(X)\to\theta^* T_E\pi_!(X)$. Soit en effet $u : A\to B$ un morphisme de $\mathbf{M}(\A)$. Soient également $y\in T(B)$, $x=T(u)(y)\in T(A)$ --- de sorte que $u$ induit un morphisme, encore noté $u$, $(A,x)\to (B,y)$ --- et $g : \Phi(B,y)\to d$ un objet de $\K(B,y)$. La deuxième hypothèse de l'énoncé implique que $f=g\circ\Phi(u) : \Phi(A,x)\to d$ est un objet de $\K(A,x)$. Le diagramme commutatif
$$\xymatrix{X(B,y)\ar[r]\ar[d]^-{X(u)} & X((B,y)\oplus d^{\oplus E\setminus e})\ar[rrr]^-{X((B,y)\oplus g^{\oplus E\setminus e})}\ar[d]^-{X(u\oplus d^{\oplus E\setminus e})} &&& X((B,y)\oplus \Phi(B,y)^{\oplus E\setminus e})\ar[d]_-{X(u\oplus\Phi(u)^{E\setminus e})} \\
X(A,x)\ar[r] & X((A,x)\oplus d^{\oplus E\setminus e})\ar[rrr]^-{X((A,x)\oplus f^{\oplus E\setminus e})} & && X((A,x)\oplus \Phi(A,x)^{\oplus E\setminus e})
}$$
fournit l'égalité $u^*\circ s_{(B,y)}=s_{(A,y)}\circ u^*$, puis $u^*\circ\sigma_{(B,y)}=\sigma_{(A,y)}\circ u^*$, d'où la fonctorialité souhaitée.

La démonstration que $\sigma$ est une section de l'effet croisé pointé associé à $(E,e)$ repose sur la quatrième hypothèse et s'établit de la même façon qu'à la proposition~\ref{pr-auxh0}.

\end{proof}

\begin{cor}\label{cor-auxhg}
 Sous les mêmes hypothèses, pour tout foncteur analytique $F\in {\rm Ob}\,\A-\mathbf{Mod}$, on dispose d'isomorphismes canoniques
$$H_*(\D;\varphi^*\pi^*F)\simeq H_*(\C;\pi^*F)\simeq {\rm Tor}^\A_*(\mathbb{Z}[T],F).$$
\end{cor}

\section{Un r\'esultat de symétrisation en homologie des foncteurs}\label{sec-mon}

Soit $A$ un mono\"ide commutatif. On note $\tilde{A}$ sa sym\'etrisation et $i_A : A\to\tilde{A}$ le morphisme de mono\"ides canonique. Autrement dit, $i_A$ est l'unité, évaluée sur $A$, de l'adjonction entre le foncteur d'oubli des groupes commutatifs vers les monoïdes commutatifs et la symétrisation (qui en est l'adjoint à gauche). Si $M$ est un ensemble muni d'une action de $A$, on note $j_M : M\to\tilde{M}$ le morphisme de $A$-ensembles initial parmi ceux dont le but est un $\tilde{A}$-ensemble (dont l'action de $A$ s'obtient par restriction par $i_A$). Explicitement, $\tilde{M}$ est le quotient de $A\times M$ par la relation d'\'equivalence identifiant $(a,m)$ et $(b,n)$ lorsqu'existe $x\in A$ tel que $x+b+m=x+a+n$ (la loi de $A$ comme son action sur $M$ sont not\'ees additivement) --- ainsi $\tilde{A}$ n'est autre que $\tilde{M}$ pour $M=A$ muni de l'action par translations ; il est muni de l'action de $\tilde{A}$ donn\'ee par $(x-y)+[a,m]=[y+a,x+m]$ (o\`u $[a,m]$ d\'esigne la classe de $(a,m)$, qu'on peut donc lire comme $-a+m$).

On note que si $M$ est le $A$-ensemble trivial $*$ (\`a un \'el\'ement), alors $\tilde{M}$ est le $\tilde{A}$-ensemble trivial $*$. La propri\'et\'e universelle de $j_M$ montre que cette application n'est autre que l'unit\'e (évaluée sur $M$) de l'adjonction entre le foncteur d'oubli des $\tilde{A}$-ensembles vers les $A$-ensembles et le foncteur $\tilde{A}\underset{A}{\times}-$ (qui lui est adjoint à gauche) ; si l'on linéarise, $j_M$ induit l'unité $\mathbb{Z}[M]\to\mathbb{Z}[\tilde{A}]\underset{\mathbb{Z}[A]}{\otimes
}\mathbb{Z}[M]\simeq\mathbb{Z}[\tilde{M}]$.

\begin{pr}\label{hom-mon}
 Les applications $i_A$ et $j_M$ induisent un isomorphisme gradu\'e $$H_*(A;\mathbb{Z}[M])\xrightarrow{\simeq}H_*(\tilde{A};\mathbb{Z}[\tilde{M}])\;;$$
 en particulier $i_A$ induit
 un isomorphisme gradu\'e $$H_*(A)\xrightarrow{\simeq}H_*(\tilde{A}).$$
\end{pr}

\begin{proof}
 La $\mathbb{Z}[A]$-alg\`ebre $\mathbb{Z}[\tilde{A}]$ est la localisation obtenue en inversant les \'el\'ements de $A$, elle est donc {\em plate}. Par cons\'equent, l'isomorphisme naturel
 $$\iota(V)\underset{\mathbb{Z}[A]}{\otimes}W\simeq V\underset{\mathbb{Z}[\tilde{A}]}{\otimes}\big(\mathbb{Z}[\tilde{A}]\underset{\mathbb{Z}[A]}{\otimes}W\big)$$
 o\`u $V$ est un $\mathbb{Z}[\tilde{A}]$-module \`a droite, $W$ un $\mathbb{Z}[A]$-module \`a gauche et $\iota : \mathbf{Mod}-\mathbb{Z}[\tilde{A}]\to\mathbf{Mod}-\mathbb{Z}[A]$ d\'esigne le foncteur d'oubli, s'\'etend en un isomorphisme gradu\'e $${\rm Tor}^{\mathbb{Z}[A]}_*(\iota(V),W)\simeq {\rm Tor}^{\mathbb{Z}[\tilde{A}]}_*\big(V,\mathbb{Z}[\tilde{A}]\underset{\mathbb{Z}[A]}{\otimes}W\big).$$
 
 La proposition est le cas particulier $V=\mathbb{Z}$ et $W=\mathbb{Z}[M]$.
\end{proof}

Dans ce qui suit, on note $\bar{\mathbb{Z}}[E]=Ker\,(\mathbb{Z}[E]\to\mathbb{Z}[*]=\mathbb{Z})$ pour tout ensenble $E$.

\begin{cor}\label{corbar}
 Il existe un morphisme naturel de complexes
 $$\xymatrix{\dots\ar[r] & \bar{\mathbb{Z}}[A]^{\otimes (n+1)}\otimes\bar{\mathbb{Z}}[M]\ar[r]\ar[d] &  \bar{\mathbb{Z}}[A]^{\otimes n}\otimes\bar{\mathbb{Z}}[M]\ar[r]\ar[d] & \dots\ar[r] & \bar{\mathbb{Z}}[M]\ar[d] \\
 \dots\ar[r] & \bar{\mathbb{Z}}[\tilde{A}]^{\otimes (n+1)}\otimes\bar{\mathbb{Z}}[\tilde{M}]\ar[r] &  \bar{\mathbb{Z}}[\tilde{A}]^{\otimes n}\otimes\bar{\mathbb{Z}}[\tilde{M}]\ar[r] & \dots\ar[r] & \bar{\mathbb{Z}}[\tilde{M}]
 }$$
qui induit un isomorphisme en homologie.
\end{cor}

\begin{proof}
 On utilise la construction barre r\'eduite (ou normalisée), qui est fonctorielle, et on applique la proposition pr\'ec\'edente.
\end{proof}

\smallskip

On s'int\'eresse maintenant \`a une g\'en\'eralisation~\guillemotleft~\`a paramètres~\guillemotright~de ce qui pr\'ec\`ede. On suppose que $\A$ est une petite cat\'egorie, $A$ un foncteur contravariant de $\A$ vers les mono\"ides ab\'eliens, et l'on note $\tilde{A} : \A^{op}\to\mathbf{Ab}$ sa symétrisation (i.e. le foncteur obtenu en sym\'etrisant au but), qui est munie d'un morphisme $i_A : A\to\tilde{A}$. Plus g\'en\'eralement, si $M$ est un foncteur contravariant de $\A$ vers les ensembles qui est muni d'une action de $A$ (i.e. on dispose d'un morphisme $A\times M\to A$ v\'erifiant les conditions \'evidentes), on peut construire un foncteur $\tilde{M}$ muni d'une action de $\tilde{A}$, avec un morphisme d'unit\'e $j_M : M\to\tilde{M}$.

Par naturalit\'e le corollaire~\ref{corbar} demeure valide dans ce contexte, dans la cat\'egorie $\mathbf{Mod}-\A$.

\begin{thm}\label{th-sym}
 Supposons que la cat\'egorie $\A$ est additive et que $F$ est un foncteur polynomial (ou m\^eme analytique) $\A\to\mathbf{Ab}$. Supposons \'egalement que $A$ et $M$ sont r\'eduits au sens o\`u leur valeur en $0$ est r\'eduite \`a un point.
 
 Alors $j_M$ induit un isomorphisme gradu\'e $${\rm Tor}^\A_*(\mathbb{Z}[M],F)\xrightarrow{\simeq}{\rm Tor}^\A_*(\mathbb{Z}[\tilde{M}],F).$$
 En particulier, $i_A$ induit un isomorphisme
 $${\rm Tor}^\A_*(\mathbb{Z}[A],F)\xrightarrow{\simeq}{\rm Tor}^\A_*(\mathbb{Z}[\tilde{A}],F).$$
\end{thm}

\begin{proof}
 On \'etablit le r\'esultat par r\'ecurrence sur le degr\'e polynomial $d$ de $F$. Plus pr\'ecis\'ement, on montre par r\'ecurrence sur $d$ que pour tout foncteur polynomial $F$ de degr\'e au plus $d$ et tout foncteur ensembliste $M$ muni d'une action du foncteur en mono\"ides $A$, la conclusion du th\'eor\`eme est valide.
 
 Si $C_\bullet$ est un complexe (avec une diff\'erentielle de degr\'e $-1$) de $\mathbf{Mod}-\A$, rappelons qu'on dispose de deux suites spectrales naturelles d'hyperhomologie telles que
 $$\mathbf{I}^1_{p,q}={\rm Tor}^\A_q(C_p,F)$$
 et
 $$\mathbf{II}^2_{p,q}={\rm Tor}_p^\A(H_q(C_\bullet),F)$$
(on indexe pour que la diff\'erentielle $d_r$ soit toujours de bidegr\'e $(-r,r-1)$) dont les aboutissements sont isomorphes.

Appliquons cela aux complexes du corollaire~\ref{corbar} : sa conclusion montre que le morphisme de complexes induit un isomorphisme gradu\'e $\mathbf{II}^2\xrightarrow{\simeq}\widetilde{\mathbf{II}}^2$, ce qui implique que le morphisme de suites spectrales $\mathbf{I}\to\widetilde{\mathbf{I}}$ induit un isomorphisme entre les aboutissements.
 
Entre les premi\`eres pages, ce morphisme
$$\mathbf{I}^1_{p,q}={\rm Tor}^\A_q(\bar{\mathbb{Z}}[A]^{\otimes p}\otimes\bar{\mathbb{Z}}[M],F)\to\tilde{\mathbf{I}}^1_{p,q}={\rm Tor}^\A_q(\bar{\mathbb{Z}}[\tilde{A}]^{\otimes p}\otimes\bar{\mathbb{Z}}[\tilde{M}],F)$$
est induit par $i_A$ est $j_M$ ; v\'erifions que l'hypoth\`ese de r\'ecurrence entraîne que c'est un isomorphisme pour $p\neq 0$.

Comme la cat\'egorie $\A$ est additive, les foncteurs diagonale $\A\to\A\times\A$ et somme directe $\oplus : \A\times\A\to\A$ sont adjoints, ils induisent donc entre groupes de torsion un isomorphisme
$${\rm Tor}^\A_*(\bar{\mathbb{Z}}[A]^{\otimes p}\otimes\bar{\mathbb{Z}}[M],F)\simeq{\rm Tor}^{\A\times\A}_*(\bar{\mathbb{Z}}[A]^{\otimes p}\boxtimes\bar{\mathbb{Z}}[M],\oplus^*F)\;;$$
et la proposition~\ref{pr-bifo} procure une suite spectrale
$$E^2_{i,j}={\rm Tor}^\A_i\big(\bar{\mathbb{Z}}[A]^{\otimes p},U\mapsto {\rm Tor}^\A_j(\bar{\mathbb{Z}}[M],F(U\oplus -))\big)$$
$$\dots\Rightarrow {\rm Tor}^{\A\times\A}_{i+j}(\bar{\mathbb{Z}}[A]^{\otimes p}\boxtimes\bar{\mathbb{Z}}[M],\oplus^*F).$$

Comme $A$ et $M$ sont r\'eduits, ce terme $E^2_{i,j}$ s'identifie, pour $p>0$, \`a $${\rm Tor}^\A_i\big(\bar{\mathbb{Z}}[A]^{\otimes p},U\mapsto {\rm Tor}^\A_j(\bar{\mathbb{Z}}[M],cr_2(F)(U,-))\big).$$
Or $cr_2(F)(U,-)$ est de degr\'e strictement inf\'erieur \`a $d$, de sorte que l'hypoth\`ese de r\'ecurrence implique que $j_M$ induit un isomorphisme 
$${\rm Tor}^\A_i\big(\bar{\mathbb{Z}}[A]^{\otimes p},U\mapsto {\rm Tor}^\A_j(\bar{\mathbb{Z}}[M],F(U\oplus -))\big)\xrightarrow{\simeq}\dots$$
$${\rm Tor}^\A_i\big(\bar{\mathbb{Z}}[A]^{\otimes p},U\mapsto {\rm Tor}^\A_j(\bar{\mathbb{Z}}[\tilde{M}],F(U\oplus -))\big).$$

En appliquant inductivement la proposition~\ref{pr-bifo} au facteur $\bar{\mathbb{Z}}[A]^{\otimes p}$ (toujours pour $p>0$) et en utilisant derechef l'hypoth\`ese de r\'ecurrence, cette fois-ci pour le foncteur $A$ agissant sur lui-m\^eme par translations, on en d\'eduit que $i_A$ et $j_M$ induisent 
un isomorphisme 
$${\rm Tor}^\A_i\big(\bar{\mathbb{Z}}[A]^{\otimes p},U\mapsto {\rm Tor}^\A_j(\bar{\mathbb{Z}}[M],F(U\oplus -))\big)\xrightarrow{\simeq}\dots
$$
$${\rm Tor}^\A_i\big(\bar{\mathbb{Z}}[\tilde{A}]^{\otimes p},U\mapsto {\rm Tor}^\A_j(\bar{\mathbb{Z}}[\tilde{M}],F(U\oplus -))\big),$$
donc par ce qui pr\'ec\`ede un isomorphisme $\mathbf{I}^1_{p,q}\to\tilde{\mathbf{I}}^1_{p,q}$ sauf peut-\^etre pour $p=0$. Mais cela contraint notre morphisme \`a \^etre \'egalement un isomorphisme pour $p=0$, puisqu'on a vu que le morphisme de suites spectrales $\mathbf{I}\to\tilde{\mathbf{I}}$ induisait un isomorphisme entre les aboutissements (son noyau et son conoyau, concentr\'es sur la ligne $p=0$ \`a la page $1$, s'y arr\^etent n\'ecessairement). Cela termine la d\'emonstration.
\end{proof}

\begin{rem}
 La méthode mise en \oe uvre dans cette démonstration est très classique en homologie des foncteurs ; par exemple, les articles \cite{Pira-pol} et \cite{FPZ} utilisent fortement la résolution barre pour les groupes abéliens dans le contexte des catégories de foncteurs pour établir des résultats sur des groupes de torsion ou d'extensions entre foncteurs polynomiaux. 
\end{rem}

\section{Homologie stable des groupes unitaires à
coefficients polynomiaux : premiers résultats}\label{scg}
$\A$ étant une petite catégorie additive, on suppose donné un {\em foncteur de dualité} $D : \A^{op}\to\A$, c'est-à-dire un foncteur vérifiant les trois conditions suivantes :
\begin{enumerate}
 \item $D$ est auto-adjoint : on dispose d'isomorphismes naturels
 $$\sigma_{X,Y} : \A(X,DY)\xrightarrow{\simeq}\A(Y,DX)\;;$$
 \item $D$ est symétrique au sens où $\sigma_{Y,X}=\sigma_{X,Y}^{-1}$ ;
 \item l'unité $Id_\A\to D^2$ (qui coïncide avec la coünité par symétrie) est un isomorphisme (en particulier, $D$ est une équivalence de catégories).
\end{enumerate}

Il s'agit du cadre général classique pour traiter d'espaces hermitiens ou symplectiques --- cf. par exemple \cite{Kn}, chap.~II, §\,2.

\begin{rem}\label{rq-exdual}
L'exemple canonique est celui où $\A=\mathbf{P}(A)$ (cf. appendice~\ref{app2}), où $A$ est un anneau muni d'une anti-involution, et $D$ la dualité usuelle (l'anti-involution permet de considérer ${\rm Hom}_A(M,A)$, où $M$ est un $A$-module à gauche, comme un $A$-module {\em à gauche}).

Nous verrons au paragraphe~\ref{sec-sco} un autre cas particulier important.
\end{rem}

On notera $\bar{a}$ pour $\sigma_{X,Y}(a)$, où $a\in\A(X,DY)$. Ainsi $\bar{\bar{a}}=a$ ; on dispose en particulier, pour tout objet $M$ de $\A$, d'une involution distinguée sur $\A(M,DM)$.

On se donne $\epsilon\in\{1,-1\}$ ({\small on pourrait faire légèrement plus général en supposant que $\A$ est $k$-linéaire, où $k$ est un anneau commutatif à involution, que $D$ est $k$-semilinéaire et que $\epsilon$ est tel que $\epsilon.\bar{\epsilon}=1$, et remplacer au but des catégories de foncteurs $\mathbf{Ab}$ par $k-\mathbf{Mod}$}) et l'on note $\mathbb{Z}_\epsilon$ la représentation $\mathbb{Z}$ de $\mathbb{Z}/2$ avec action de l'élément non trivial par $\epsilon$.

On note $TD^2$ le foncteur $\A^{op}\to\mathbf{Ab}\quad M\mapsto\A(M,DM)$, qui est donc muni d'une action de $\mathbb{Z}/2$. On en définit un sous-foncteur $\Gamma D^2_\epsilon$ et un quotient $SD^2_\epsilon$ par
$$\Gamma D^2_\epsilon={\rm Hom}_{\mathbb{Z}/2}(\mathbb{Z}_\epsilon,TD^2)\qquad\text{et}\qquad SD^2_\epsilon=\mathbb{Z}_\epsilon\underset{\mathbb{Z}/2}{\otimes}TD^2$$
(de sorte que $\Gamma D^2(M)=\{f\in TD^2(M)\,|\,\bar{f}=\epsilon f\}$) ; on dispose d'un morphisme de norme
$$SD^2_\epsilon\to\Gamma D^2_\epsilon\qquad 1\otimes x\mapsto x+\epsilon\bar{x}$$
dont l'image est notée $\bar{S} D^2_\epsilon$.

Si $2$ est inversible dans $\A$ (i.e. dans chaque groupe abélien $\A(M,N)$), alors ces trois foncteurs coïncident.

Afin de traiter en même temps de ces différents foncteurs et d'éventuelles variantes, nous introduisons les hypothèses suivantes.

\begin{hyp}\label{hyp-soutruc}
On suppose que $T : \A^{op}\to\mathbf{Ab}$ est un foncteur muni d'un morphisme $T\to \Gamma D^2_\epsilon$ dont le noyau est additif\,\footnote{Un foncteur additif est un foncteur polynomial de degré au plus $1$ et nul en $0$, c'est-à-dire un foncteur commutant aux sommes directes finies.}.
\end{hyp}

Cette hypothèse s'applique en particulier à $SD^2_\epsilon$ muni de la norme.

On utilisera aussi l'hypothèse plus faible suivante :

\begin{hyp}\label{hyp-soutruc2} On suppose que $T$ est un foncteur contravariant de $\A$ vers les monoïdes commutatifs réguliers, dont la symétrisation vérifie l'hypothèse~\ref{hyp-soutruc}.
\end{hyp}

(Il revient au même de demander que $T$ soit un sous-foncteur en monoïdes d'un foncteur vérifiant~\ref{hyp-soutruc}.)

\begin{defi}\label{df-fh}
Supposons que $T$ vérifie l'hypothèse~\ref{hyp-soutruc2}.

 On appelle $T$-forme hermitienne sur un objet $M$ de $\A$ tout élément de $T(M)$. Une telle forme est dite non dégénérée si l'élément de $\A(M,DM)$ associé est un isomorphisme.
 
 On note $\A^T$, selon nos conventions générales, la catégorie des objets de $\A$ munis d'une forme $T$-hermitienne. On rappelle que les morphismes $(M,x)\to (N,y)$ de $\A^T$ sont les morphismes $f : M\to N$ de $\A$ tels que $T(f)(y)=x$. On munit la catégorie $\A^T$ de la structure monoïdale symétrique (notée $\overset{T}{\oplus}$) induite par la somme directe de $\A$ et la structure monoïdale de $T$ (cf. §\,\ref{s-mos}).

On note $\mathbf{H}^T(\A)$ la sous-catégorie pleine de $\A^T$ dont les objets sont les espaces $T$-hermitiens non dégérénés (i.e. les objets de $\A$ munis d'une forme $T$-hermitienne non dégénérée). C'est une sous-catégorie monoïdale de $\A^T$.
\end{defi}

\begin{ex}[Cas fondamental]
Soit $A$ est un anneau muni d'une anti-involution et $\A=\mathbf{P}(A)$ la catégorie additive avec dualité correspondante (cf. remarque~\ref{rq-exdual}). Si $M$ est un objet de $\mathbf{P}(A)$, un élément de $SD^2_\epsilon(M)$ est exactement une forme hermitienne (éventuellement dégénérée ; si $A$ est commutatif et l'involution triviale c'est donc une forme quadratique) au sens usuel sur $M$, tandis qu'un élément de $\Gamma D^2_\epsilon(M)$ est exactement une forme sesquilinéaire sur $M$ (notions qu'on peut identifier lorsque $2$ est inversible dans $A$, puisque la norme est alors un isomorphisme). Le foncteur $\bar{S}D_{-1}$ correspond quant à lui à la notion usuelle de forme symplectique (toujours éventuellement dégénérée).
\end{ex}

On dispose d'un foncteur d'oubli $\pi^T : \A^T\to\A$ (noté souvent simplement $\pi$) ; sa restriction à $\mathbf{H}^T(\A)$ se relève en un foncteur vers $\mathbf{S}(\A)$ (catégorie définie en \ref{nism} en fin d'introduction) de la façon suivante : soient $u : M\to N$ est un morphisme de $\mathbf{H}^T(\A)$, $r$ et $s$ les éléments de $\A(M,DM)$ et $\A(N,DN)$ respectivement associés aux formes sur $M$ et $N$, on associe $(r^{-1}\bar{u}s,u)\in\mathbf{S}(\A)(M,N)$ à $u$.


\begin{pr}\label{prhs1e}
Sous l'hypothèse~\ref{hyp-soutruc2}, la catégorie monoïdale symétrique $\mathbf{H}^T(\A)$ vérifie les conditions de la section~\ref{sec1}.
\end{pr}

\begin{proof}
 Par la remarque~\ref{rq-sthyp}, on peut supposer, pour simplifier, que l'hypothèse forte~\ref{hyp-soutruc} est satisfaite.

La structure fonctorielle sur les groupes d'automorphismes s'inspire du foncteur vers $\mathbf{S}(\A)$ : avec les notations précédentes, le morphisme ${\rm Aut}_{\mathbf{H}^T(\A)}(M)\to {\rm Aut}_{\mathbf{H}^T(\A)}(N)$ qu'induit $u$ est donné par $f\mapsto ufr^{-1}\bar{u}s+1-ur^{-1}\bar{u}s$. Pour montrer que ces morphismes préservent bien la forme sur $N$, on utilise l'hypothèse sur le noyau de $T\to\Gamma D^2_\epsilon$ (qui permet essentiellement de se ramener au cas où ce morphisme est injectif\,\footnote{Le passage du cas de $\bar{S}D^2_\epsilon$ au cas général dans le lemme~\ref{lm-scf} ci-après, par exemple, peut se traiter de la même façon.}) : on a une identité quadratique
\begin{equation}\label{eq-quad}
 T(a+b)(x)=T(a)(x)+T(b)(x)+\theta(\bar{b}\tilde{x}a+\epsilon\bar{a}\tilde{x}b)
\end{equation}
(pour une fonction additive $\theta$ convenable ; on désigne par $\tilde{x}$ l'image de $x\in T(A)$ dans $\A(A,DA)$) ; elle montre que les morphismes $T(N)\to T(N)$ induits par $ufr^{-1}\bar{u}s+1-ur^{-1}\bar{u}s$ et $ur^{-1}\bar{u}s+1-ur^{-1}\bar{u}s=1$ sont les mêmes.

En conservant les mêmes notations pour un morphisme $u$ de $\mathbf{H}^T(\A)$, l'endomorphisme $\phi$ de $\pi M\oplus \pi N$ donné par la matrice
$$\left(\begin{array}{cc}
 0 & r^{-1}\bar{u}s \\
 u & 1-ur^{-1}\bar{u}s
\end{array}\right)
$$
est en fait un automorphisme involutif de $M\overset{T}{\oplus}N$ (dans $\mathbf{H}^T(\A)$) comme on le vérifie de façon analogue à ce qui précède, en utilisant l'identité quadratique~(\ref{eq-quad}), et il fait commuter le diagramme
\begin{equation}\label{eq-transt}
\xymatrix{M\ar[r]^-u\ar[rrd] & N\ar[r] & M\overset{T}{\oplus} N \\
& & M\overset{T}{\oplus} N\ar[u]_-\phi
}
\end{equation}
d'où la transitivité stable requise.

Enfin, l'hypothèse (H forte) est satisfaite, même sous une forme légèrement plus forte : soient $H$ un objet de $\mathbf{H}^T(\A)$, $r$ l'élément de $\A(M,DM)$ sous-jacent, $M$ et $N$ deux objets de $\A^T$. Tout morphisme $M\overset{T}{\oplus} H\to N\overset{T}{\oplus} H$ faisant commuter le diagramme
$$\xymatrix{H\ar[r]\ar[rd] & M\overset{T}{\oplus} H\ar[d] \\
& N\overset{T}{\oplus} H
}$$
est donné par une matrice $\left(\begin{array}{cc} u & 0\\
                                  f & 1
                                 \end{array}\right)$
qui préserve les formes, ce qui implique $rf=0$, i.e. $f=0$ puisque $r$ est inversible, donc que notre morphisme provient d'un morphisme $M\to N$.
\end{proof}

Par conséquent, la proposition~\ref{pr-dvks} et le corollaire~\ref{cor-dvks} entraînent :
\begin{pr}\label{pr-ksh1}
 Sous l'hypothèse~\ref{hyp-soutruc2}, il existe des isomorphismes
$$H^{st}_*(\mathbf{H}^T(\A);F)\simeq  H_*\big(\mathbf{H}^T(\A)\times {\rm U}^T_\infty(\A);F\big)\;,$$
$$H^{st}_n(\mathbf{H}^T(\A);F)\simeq\bigoplus_{p+q=n}{\rm Tor}^{\mathbf{H}^T(\A)}_p\big(H_q({\rm U}^T_\infty(\A);\mathbb{Z}),F\big)$$
(où le membre de gauche du groupe de torsion est un foncteur constant)
 naturels en $F\in {\rm Ob}\,\mathbf{H}^T(\A)-\mathbf{Mod}$, où ${\rm U}^T_\infty(\A)$ désigne la colimite des groupes d'automorphismes d'objets de $\mathbf{H}^T(\A)$ prise selon une tranche (cf. §\,\ref{sec1})
et des isomorphismes (non nécessairement naturels)
$$H^{st}_n(\mathbf{H}^T(\A);F)\simeq\bigoplus_{p+q=n}H_p\big({\rm U}^T_\infty(\A);H_q(\mathbf{H}^T(\A);F)\big).$$
\end{pr}

Lorsque $T=SD^2_\epsilon$ ou $\bar{S}D^2_\epsilon$, on notera ${\rm U}_{\infty,\infty}(\A,D,\epsilon)$ pour ${\rm U}^T_\infty(\A)$ car on peut alors prendre une tranche constituée d'espaces hyperboliques, comme cela résultera du lemme~\ref{lm-scf} (la notion d'espace hyperbolique est rappelée ci-dessous) --- dans le cas où $\A=\mathbf{P}(A)$, où $A$ est un anneau muni d'une anti-involution, et où l'on prend pour $D$ la dualité usuelle et $\epsilon=1$, un choix distingué de tranche conduit à ${\rm U}_{\infty,\infty}(A)=\underset{n\in\mathbb{N}}{\col}U_{n,n}(A)$, tandis que pour $\epsilon=-1$ on obtient le groupe symplectique infini $Sp_\infty(A)$.

\paragraph*{Espaces hyperboliques}
Soit $V$ un objet de $\A$ ; supposons ici que l'on se trouve dans le {\em cas générique}, i.e. que $T$ est le foncteur $\bar{S}D^2_\epsilon$ ou $SD^2_\epsilon$ (muni de la flèche canonique vers $\Gamma D^2_\epsilon$). On note $\tilde{V}$ l'objet $V\oplus DV$ de $\A$ muni de la forme correspondant dans la décomposition
$$T(V\oplus DV)\simeq T(V)\oplus T(DV)\oplus\A(DV,DV)$$
à l'identité de $DV$. L'élément de $\A(V\oplus DV,D(V\oplus DV))\simeq\A(V\oplus DV,DV\oplus V)$ correspondant a pour matrice $\left(\begin{array}{cc} 0 & 1\\                                                                                                            \epsilon & 0
\end{array}\right)
$, il est donc inversible : $\tilde{V}$ est un objet de $\mathbf{H}^T(\A)$, appelé {\em espace hyperbolique} associé à $V$. (Cette construction ne définit {\em pas} un foncteur de $\A$ dans $\mathbf{H}^T(\A)$, puisque $D$ est contravariant, mais seulement un foncteur de la sous-catégorie des isomorphismes de $\A$ vers $\mathbf{H}^T(\A)$). Il s'agit d'une généralisation classique et très naturelle des espaces hermitiens (au sens usuel) hyperboliques qui correspondent au cas $\A=\mathbf{P}(A)$ déjà mentionné.

Dans la section suivante, on utilise de façon essentielle les propriétés fondamentales des espaces hyperboliques pour relier l'homologie stable des groupes unitaires à coefficients polynomiaux à des groupes d'homologie des foncteurs dans le cas générique (au sens mentionné précédemment).

Le cas non générique, plus délicat (traité partiellement dans la dernière section), ne peut plus être abordé de la même manière. Le cas typique est fourni en considérant, sur $\mathbf{P}(A)$, où $A$ est un anneau commutatif (avec l'involution triviale et $\epsilon=1$), les groupes orthogonaux \guillemotleft~euclidiens~\guillemotright~$O_n(A)$ (avec le choix de $T$ y afférent --- cf. section~\ref{sng} ; cette situation peut évidemment se réléver générique dans certains cas particuliers, comme lorsque $-1$ est somme de carrés dans $A$).

\section{Homologie stable des groupes unitaires dans le cas générique}\label{seccen}

\subsection{Le résultat principal}\label{par-dvg}

On considère encore, dans ce paragraphe, une petite catégorie additive avec dualité $(\A,D)$ et $\epsilon\in\{-1,1\}$.

\begin{hyp}\label{hyp-gener}
 On suppose dans ce paragraphe que le foncteur $T$ est soit $\bar{S} D^2_\epsilon$ soit $SD^2_\epsilon$ (munis de la flèche canonique vers $\Gamma D^2_\epsilon$).
 \end{hyp}

On note dans la suite $\mathbf{M}\A^T$ la sous-catégorie de $\A^T$ des morphismes $f$ tels que $\pi(f)\in\mathbf{M}(\A)$ (elle contient en particulier $\mathbf{H}^T(\A)$).

\begin{lm}\label{lm-scf}
 Soit $V$ un objet de $\A$. L'espace hyperbolique associé $\tilde{V}\in {\rm Ob}\,\mathbf{H}^T(\A)$ vérifie les propriétés suivantes.
 \begin{enumerate}
  \item\label{lfp1} L'ensemble $\mathbf{M}\A^T((V,x),\tilde{V})$ est non vide pour tout $x\in T(V)$. De surcroît, on peut choisir un élément de cet ensemble qui possède une rétraction dans $\A$ indépendante du choix de $x$.
  \item\label{lfp2} On peut de plus trouver un élément $\eta$ de $\mathbf{M}\A^T((V,0),\tilde{V})$ tel que tout morphisme de $\mathbf{M}\A^T$ de source $(V,0)$ et de but non dégénéré se factorise par $\eta$.
  \item\label{lfp3} Plus généralement, si $\varphi : (V,0)\to (E,r)$ est un morphisme de $\A^T$ tel que $r\varphi : V\to DE$ soit un monomorphisme scindé de $\A$, alors $\varphi$ se factorise par $\eta$.
  \item\label{lfp4} Soient $x\in T(V)$, $m : V\to W$ un monomorphisme scindé de $\A$ et $\phi : (W,0)\to E$ un morphisme de $\mathbf{M}\A^T$ avec $E$ non dégénéré. Il existe un diagramme commutatif de $\mathbf{M}\A^T$ de la forme
  $$\xymatrix{(V,x)\ar[rd]\ar[r]^-{1\oplus m} & (V,x)\overset{T}{\oplus} (W,0)\ar[r]^-{1\oplus\phi}  & (V,x)\overset{T}{\oplus}E\\
  & \tilde{V}\ar[ru] &
}$$
 \end{enumerate}
\end{lm}

\begin{proof}

 Soient $x\in T(V)$ et $t\in\A(V,DV)$ un élément relevant $x$ (on utilise les épimorphismes canoniques $TD^2\twoheadrightarrow SD^2_\epsilon\twoheadrightarrow\bar{S} D^2_\epsilon$ ; autrement dit, dans le cas de $\bar{S}D^2_\epsilon$, $x=t+\epsilon\bar{t}$). Alors le morphisme $V\to V\oplus DV$ de $\A$ de composantes $1_V$ et $t$ est un monomorphisme, scindé par la projection sur $V$ (qui ne dépend pas de $x$), qui définit un morphisme 
 $(V,x)\to\tilde{V}$ dans $\A^T$, d'où le premier point.


Prouvons le troisième (donc le deuxième) point dans le cas où $T=\bar{S} D^2_\epsilon$, par exemple. Soit $\alpha : DE\to V$ une rétraction dans $\A$ à $r\varphi$. Il existe $t\in\A(DV,V)$ tel que $\alpha r\bar{\alpha}=\epsilon t+\bar{t}$. Alors le morphisme $f : V\oplus DV\to E$ de composantes $\varphi$ et $\epsilon:\bar{\alpha}-\varphi t$ vérifie $f\eta=\varphi$ (où $\eta$ est le morphisme construit précédemment pour la forme nulle, i.e. l'inclusion $V\to V\oplus DV$) et préserve les formes car
$$\bar{f}rf=\left(\begin{array}{cc}\bar{\varphi} r\varphi & \epsilon\bar{\varphi}r\alpha-\bar{\varphi}r\varphi t \\
                   \epsilon\alpha r\varphi-\bar{t}\bar{\varphi}r\varphi & (\epsilon\alpha-\bar{t}\bar{\varphi})r(\epsilon\bar{\alpha}-\varphi t)
                  \end{array}
\right)$$
qui égale $\left(\begin{array}{cc} 0 & 1\\
                  \epsilon & 0
                 \end{array}
\right)$
parce que $\alpha r\varphi=1_V$, $\bar{\varphi}r\varphi=0$ et $\alpha r\bar{\alpha}=\epsilon t+\bar{t}$.

Pour le dernier point, l'assertion~\ref{lfp3} permet de se ramener au cas où $m=Id_V$ et où $\phi$ est le morphisme vers $\tilde{V}$ construit précédemment. En ce cas, si $\varphi : (V,x)\to\tilde{V}$ et $\eta : (V,0)\to\tilde{V}$ sont les morphismes construits en début de démonstration et $\psi : \tilde{V}\to V$ leur rétraction commune dans $\A$, on vérifie aisément que le morphisme $\alpha : \tilde{V}\to (V,x)\overset{T}{\oplus}\tilde{V}$ de composantes $\psi$ et $1+(\varphi-\eta)\psi$ (qui est scindé par le morphisme de composantes $\eta-\varphi$ et $1$) préserve les formes, et il fait commuter le diagramme
$$\xymatrix{(V,x)\ar[dr]_\eta\ar[rr]^{\left(\begin{array}{cc}1\\ \varphi\end{array}\right)} & & (V,x)\overset{T}{\oplus}\tilde{V}  \\
& \tilde{V}\ar[ru]_\alpha &
}$$
ce qui achève la démonstration.
\end{proof}

Nous pouvons maintenant établir rapidement le résultat principal de cette section.

\begin{thm}\label{th-dvg}
 Sous l'hypothèse~\ref{hyp-gener}, pour tout foncteur polynomial $F\in {\rm Ob}\,\A-\mathbf{Mod}$, l'application canonique
 $$H_*(\mathbf{H}^T(\A);i^*\pi^*F)\to H_*(\A^T;\pi^*F)\simeq {\rm Tor}^\A_*(\mathbb{Z}[T],F)$$
 est un isomorphisme, où $i : \mathbf{H}^T(\A)\to\A^T$ désigne le foncteur d'inclusion et, comme plus haut, $\pi : \A^T\to\A$ désigne le foncteur d'oubli.
\end{thm}

\begin{proof}
On applique le corollaire~\ref{cor-auxhg} aux foncteurs $\mathbf{H}^T(\A)\xrightarrow{i}\A^T\xrightarrow{\pi}\A$, avec le choix suivant : $\B=\mathbf{M}\A^T$ et $\Phi : \mathbf{H}^T(\A)\to\mathbf{H}^T(\A)$ est le foncteur composé de $\pi$ et de $V\mapsto (V,0)\quad \A\to\A^T$. Les deux premières hypothèses de la proposition~\ref{pr-auxhg} sont trivialement satisfaites.

La propriété~\ref{lfp2} du lemme~\ref{lm-scf} montre que l'hypothèse de connexité requise est également vérifiée (les catégories concernées ont même un objet pseudo-initial). La propriété~\ref{lfp4} du même lemme montre la validité de la dernière condition (de factorisation) nécessaire à l'application de la proposition~\ref{pr-auxhg} (ou du corollaire~\ref{cor-auxhg}). Cela termine la démonstration. 
\end{proof}

Combinant le théorème~\ref{th-dvg} à la proposition~\ref{pr-ksh1}, on obtient :
\begin{cor}\label{corf} Sous les mêmes hypothèses, il existe des isomorphismes naturels
 $$H^{st}_n(\mathbf{H}^T(\A);F)\simeq\bigoplus_{p+q=n}{\rm Tor}^{\A}_p\big(H_q({\rm U}^T_\infty(\A);\mathbb{Z})\otimes\mathbb{Z}[T],F\big)$$
et des isomorphismes (non nécessairement naturels)
$$H^{st}_n(\mathbf{H}^T(\A);F)\simeq\bigoplus_{p+q=n}H_p\big({\rm U}^T_\infty(\A);{\rm Tor}^{\A}_q(\mathbb{Z}[T],F)\big)$$
(où l'action de ${\rm U}^T_\infty(\A)$ est triviale).
\end{cor}

Dans le cas où $\A=\mathbf{P}(k)$, $k$ étant un corps fini, l'homologie à coefficients constants de la même caractéristique que $k$ de ${\rm U}^T_\infty(\A)$ est triviale, de sorte qu'on retrouve exactement le résultat principal de \cite{DV}.

\subsection{Cas particulier : les résultats de Scorichenko sur la $K$-théorie stable}\label{sec-sco}

Supposons ici simplement que $\A$ est une petite catégorie additive (sans dualité) et munissons la petite catégorie additive $\B=\A^{op}\times\A$ de sa dualité canonique, échangeant les deux facteurs. Si $x=(B,A)$ est un objet de $\B$, un morphisme $x\to Dx$ est la donnée de deux morphismes $A\to B$ de $\A$ ; l'action de $\mathbb{Z}/2$ consiste à les intervertir. Prenant $\epsilon=1$, on en déduit que les trois foncteurs $SD^2$, $\Gamma D^2$ et $\bar{S} D^2$ sont isomorphes et l'on obtient que la catégorie $\B^T$ correspondante est (isomorphe à) la {\em catégorie des factorisations} $\mathbf{F}(\A)$ (dont la définition est rappelée en appendice~\ref{sagl}). Quant aux formes non dégénérées, ce sont celles correspondant à une flèche $A\to B$ qui est un isomorphisme. On en déduit que la catégorie $\mathbf{Q}^T(\B)$ correspondante est équivalente à la catégorie $\mathbf{S}(\A)$ (qui en est la sous-catégorie correspondant au cas où la flèche est l'identité).

Par conséquent, la proposition~\ref{pr-ksh1} se spécialise comme suit :

\begin{pr}\label{presco}
 Étant donné un objet $X$ de $\mathbf{S}(\A)-\mathbf{Mod}$, il existe des isomorphismes naturels
$$H_*({\rm GL}_\infty(\A);X_\infty)\simeq H_*(\mathbf{S}(\A)\times {\rm GL}_\infty(\A);X)$$
où ${\rm GL}_\infty(\A)$ désigne la colimite des groupes d'automorphismes d'objets de $\A$ selon une tranche (cf. section~\ref{sec1}) et $X_\infty$ la colimite des $X(a)$ selon la même tranche. 
\end{pr}

(Pour $\A=\mathbf{P}(A)$, où $A$ est un anneau, un choix évident de tranche donne le groupe usuel $GL_\infty(A)$.)

Le théorème~\ref{th-dvg} permet d'en déduire, compte-tenu de l'isomorphisme~(\ref{eq-hofac}) de l'appendice :

\begin{thm}[Scorichenko]\label{th-scok}
 Si $F\in {\rm Ob}\,(\A^{op}\times\A)-\mathbf{Mod}$ est analytique, on dispose d'un isomorphisme naturel
$$H_*({\rm GL}_\infty(\A);F_\infty)\simeq HH_*(\A\times {\rm GL}_\infty(\A);F).$$
\end{thm}

Pour l'équivalence entre cet énoncé et l'énoncé originel de Scorichenko --- qui supposait $\A=\mathbf{P}(A)$ pour un anneau $A$, mais ses arguments s'appliquent sans changement au cas général --- en terme de $K$-théorie stable, voir \cite{DV}, §\,2.3. Rappelons également que le cas où $A$ est un corps fini avait été traité d'abord par Betley et (indépendamment) Suslin (cf. l'appendice de \cite{FFSS}).

\begin{rem}
 La démonstration originelle de Scorichenko (cf. \cite{Sco}\,\footnote{L'ensemble de la stratégie de Scorichenko et une partie significative de sa mise en \oe uvre sont présentés dans l'ouvrage publié \cite{FFPS} (article de Franjou-Pirashvili).}) fonctionne selon un principe très analogue, mais transite par la sous-catégorie $\mathbf{M}(\A)$ (ou l'équivalent dual, la sous-catégorie des épimorphismes scindés) plutôt que par $\mathbf{S}(\A)$. La méthode proposée ici est encore plus directe : on n'a besoin d'{\em aucune} hypothèse polynomiale dans la proposition~\ref{presco}, tandis que Scorichenko utilise deux fois des arguments nécessitant le caractère polynomial (une partie de sa comparaison nécessite que le bifoncteur soit polynomial relativement à la variable contravariante tandis que l'autre nécessite qu'il le soit relativement à la variable covariante).
\end{rem}

\subsection{Exemples de calculs}

On se donne de nouveau une petite catégorie additive avec dualité $(\A,D)$, $\epsilon\in\{-1,1\}$, et $T$ vérifiant l'hypothèse~\ref{hyp-gener}.

Dans ce paragraphe, on rappelle, suivant \cite{DV} (§\,4), comment calculer les groupes de torsion apparaissant au théorème~\ref{th-dvg} (et donc l'homologie de ${\rm U}_{\infty,\infty}(\A,D,\epsilon)$ à coefficients tordus polynomiaux, supposant connue son homologie à coefficients constants), au moins lorsque $2$ est inversible dans $\A$ et qu'on sait calculer des groupes de torsion sur $\A$ entre foncteurs polynomiaux.

\begin{pr}[Cf. \cite{DV}]\label{pr-invol}
\begin{enumerate}
 \item  Il existe un isomorphisme gradué
$${\rm Tor}^\A_*(\mathbb{Z}[TD^2],F)\simeq HH_*(\A;\sigma^*F)$$
naturel en $F\in {\rm Ob}\,\A-\mathbf{Mod}$, où $\sigma : \A^{op}\times\A\to\A$ désigne le foncteur $(M,N)\mapsto DM\oplus N$.
\item L'involution du terme de gauche induite par celle de $TD^2$ correspond à droite à
$$HH_*(\A;\sigma^*F)\simeq HH_*(\A^{op};\tau^*\sigma^*F)\xrightarrow{D_*}HH_*(\A;\sigma^*F)$$
où $\tau : \A\times\A^{op}\to\A^{op}\times\A$ désigne l'échange des facteurs (et où l'on a utilisé l'isomorphisme $\sigma\tau\simeq\sigma D$ déduit de $D^2\simeq Id$).
\item L'isomorphisme est compatible aux coproduits externes en ce sens que le diagramme
$$\xymatrix{{\rm Tor}^\A_*(\mathbb{Z}[TD^2],F\otimes G)\ar[d]\ar[r]^\simeq & HH_*(\A;\sigma^*(F\otimes G))\ar[dd]\\
{\rm Tor}^\A_*(\mathbb{Z}[TD^2]\otimes\mathbb{Z}[TD^2],F\otimes G)\ar[d] & \\
{\rm Tor}^\A_*(\mathbb{Z}[TD^2],F)\otimes {\rm Tor}^\A_*(\mathbb{Z}[TD^2],G)\ar[r]^-\simeq & HH_*(\A;\sigma^*F)\otimes HH_*(\A;\sigma^*G)
}$$
commute, où la flèche verticale en haut à gauche est induite par la diagonale et les deux autres flèches verticales sont les coproduits.
\end{enumerate}
\end{pr}

Ce résultat s'établit exactement de la même façon que dans le cas particulier des espaces vectoriels sur un corps fini, traité dans {\em op. cit.} ; plutôt que de l'illustrer par des calculs sur des foncteurs exponentiels comme dans \cite{DV}, on s'intéresse ici surtout aux puissances tensorielles.


\begin{cor}\label{cor-tens}
 Soit $A : \A\to\mathbf{Ab}$ un foncteur additif prenant des valeurs $\mathbb{Z}$-plates.
\begin{enumerate}
 \item On a
$${\rm Tor}^\A_*(\mathbb{Z}[TD^2],A^{\otimes 2n+1})=0$$
et
$${\rm Tor}^\A_*(\mathbb{Z}[TD^2],A^{\otimes 2n})\simeq {\rm Tor}^\A_*(A^\vee,A)^{\otimes n}\underset{\Sigma_n}{\otimes}\mathbb{Z}[\Sigma_{2n}]$$
où $(-)^\vee$ désigne la précomposition par $D$ ; le groupe symétrique $\Sigma_n=\mathfrak{S}(\mathbf{n})$ (où $\mathbf{i}=\{1,\dots,i\}$) est plongé dans $\Sigma_{2n}\simeq\mathfrak{S}(\mathbf{2}\times\mathbf{n})$ par $\sigma\mapsto id_\mathbf{2}\times\sigma$.
Cet isomorphisme est $\Sigma_{2n}$-équivariant et l'involution est induite à droite par $u\times id_\mathbf{n}$, où $u$ est l'élément non trivial de $\Sigma_2$.
\item Supposons que $2$ est inversible dans $\A$. Alors
$${\rm Tor}^\A_*(\mathbb{Z}[SD^2_\epsilon],A^{\otimes 2n+1})=0$$
et
$${\rm Tor}^\A_*(\mathbb{Z}[SD^2_\epsilon],A^{\otimes 2n})\simeq {\rm Tor}^\A_*(A^\vee,A)^{\otimes n}\underset{\Sigma_2^n\rtimes\Sigma_n}{\otimes}\mathbb{Z}[\Sigma_{2n}]$$
où $\Sigma_n$ opère comme précédemment et chaque facteur $\Sigma_2$ opère sur le facteur ${\rm Tor}^\A_*(A^\vee,A)$ correspondant par $\epsilon$ fois l'involution de ${\rm Tor}^\A_*(A^\vee,A)$ et à droite par $(\tau_1,\dots,\tau_n)\mapsto ((i,j)\mapsto (\tau_j(i),j))$.

Cet isomorphisme est $\Sigma_{2n}$-équivariant.
\end{enumerate}
\end{cor}

\begin{proof}
 Tout repose sur la formule du binôme vue comme isomorphisme naturel $\Sigma_d$-équivariant
$$(U\oplus V)^{\otimes d}\simeq\bigoplus_{i+j=d}(U^{\otimes i}\otimes V^{\otimes j})\underset{\Sigma_i\times\Sigma_j}{\otimes}\mathbb{Z}[\Sigma_d],$$
où $U$ et $V$ sont des groupes abéliens.

On déduit d'abord de cette formule et de la proposition~\ref{pr-invol} un isomorphisme naturel $\Sigma_d$-équivariant
$${\rm Tor}^\A_*(\mathbb{Z}[TD^2],F)\simeq \bigoplus_{i+j=d}HH_*(\A;(A^\vee)^{\otimes i}\boxtimes A^{\otimes j})\underset{\Sigma_i\times\Sigma_j}{\otimes}\mathbb{Z}[\Sigma_d]$$
soit, en utilisant l'hypothèse de platitude,
$${\rm Tor}^\A_*(\mathbb{Z}[TD^2],F)\simeq \bigoplus_{i+j=d}{\rm Tor}_*^\A((A^\vee)^{\otimes i},A^{\otimes j})\underset{\Sigma_i\times\Sigma_j}{\otimes}\mathbb{Z}[\Sigma_d].$$

La formule du binôme procure également des isomorphismes naturels
$${\rm Tor}^\A_*(F\otimes G,A^{\otimes j})\simeq\bigoplus_{p+q=j}({\rm Tor}^\A_*(F,A^{\otimes p})\otimes {\rm Tor}^\A_*(G,A^{\otimes q}))\underset{\Sigma_p\times\Sigma_q}{\otimes}\mathbb{Z}[\Sigma_j]$$
d'où l'on déduit que ${\rm Tor}^\A_*((A^\vee)^{\otimes i},A^{\otimes j})$ est nul pour $i\neq j$ et isomorphe à ${\rm Tor}^\A_*(A^\vee,A)^{\otimes i}\otimes\mathbb{Z}[\Sigma_i]$ pour $i=j$. On tire des deux isomorphismes précédents la première assertion, sauf peut-être l'involution.

Tout le reste de la proposition s'en déduit en suivant explicitement ladite involution dans chaque isomorphisme, et en utilisant, lorsque $2$ est inversible dans $\A$, le scindage $TD^2\simeq SD^2_1\oplus SD^2_{-1}$ qui induit $\mathbb{Z}[TD^2]\simeq\mathbb{Z}[SD^2_1]\otimes\mathbb{Z}[SD^2_{-1}]$, d'où l'on tire par ce qui précède un isomorphisme naturel
$${\rm Tor}^\A_*(\mathbb{Z}[TD^2],A^{\otimes 2n})\simeq\bigoplus_{i+j=n}{\rm Tor}^\A_*(\mathbb{Z}[SD^2_1],A^{\otimes 2i})\otimes {\rm Tor}^\A_*(\mathbb{Z}[SD^2_{-1}],A^{\otimes 2j})\underset{\Sigma_{2i}\times\Sigma_{2j}}{\otimes}\mathbb{Z}[\Sigma_{2n}].$$

On commence par ${\rm Tor}^\A_*(\mathbb{Z}[TD^2],A^{\otimes 2})$ : l'identification de l'involution est directe en reprenant la démonstration de la proposition~\ref{pr-bifo}. On en déduit alors facilement, lorsque $2$ est inversible dans $\A$, que la décomposition
$${\rm Tor}^\A_*(\mathbb{Z}[TD^2],A^{\otimes 2})\simeq {\rm Tor}^\A_*(\mathbb{Z}[SD^2_1],A^{\otimes 2})\oplus {\rm Tor}^\A_*(\mathbb{Z}[SD^2_{-1}],A^{\otimes 2})$$
s'obtient comme annoncé.

Pour le cas général, on note que l'isomorphisme obtenu dans le premier point s'obtient à partir du cas précédent et de coproduits : le coproduit
$${\rm Tor}^\A_*(\mathbb{Z}[TD^2],A^{\otimes 2n})\simeq {\rm Tor}^\A_*(\mathbb{Z}[TD^2],(A^{\otimes 2})^{\otimes n})\to {\rm Tor}^\A_*(\mathbb{Z}[TD^2],A^{\otimes 2})^{\otimes n}$$
$$\simeq ({\rm Tor}^\A_*(A^\vee,A)\otimes\mathbb{Z}[\Sigma_2])^{\otimes n}\simeq {\rm Tor}^\A_*(A^\vee,A)^{\otimes n}\otimes\mathbb{Z}[\Sigma_2^n]$$
est $\Sigma_2^n\rtimes\Sigma_n$-équivariant ; notre isomorphisme est le morphisme $\Sigma_{2n}$-équivariant qui s'en déduit par adjonction entre restriction et (co)induction des $\Sigma_2^n\rtimes\Sigma_n$-modules aux $\Sigma_{2n}$-modules. Il est alors facile de conclure --- une manière de procéder consiste à exploiter les formules précédentes pour munir ${\rm Tor}^\A_*(\mathbb{Z}[TD^2],A^{\otimes 2\bullet})$ d'une structure d'algèbre de Hopf bigraduée : le cas particulier fondamental susmentionné permet d'identifier l'involution et les facteurs correspondant à $SD^2_1$ et $SD^2_{-1}$ sur les éléments primitifs, un argument de convolution permet d'en déduire une identification complète\,\footnote{Il s'agit d'un raisonnement tout à fait analogue à celui mené dans \cite{DV}, §\,4, pour des structures de Hopf similaires déduites de structures exponentielles sur les foncteurs.}.

\end{proof}


En combinant ce résultat à la proposition~\ref{pr-ksh1}, au théorème~\ref{th-dvg} et au corollaire~\ref{cor-chpz}, on obtient le calcul suivant :

\begin{cor}\label{cor-azp}
 Soient $A$ un anneau commutatif sans torsion (sur $\mathbb{Z}$) où $2$ est inversible et $n\in\mathbb{N}$. 
\begin{itemize}
 \item On a $$\underset{i\in\mathbb{N}}{\col}H_*(O_{i,i}(A);(A^{2i})^{\otimes 2n+1})=0.$$
\item On a $$\underset{i\in\mathbb{N}}{\col}H_*(O_{i,i}(A);(A^{2i})^{\otimes 2n})\simeq H_*\big(O_{\infty,\infty}(A);(A\otimes M_*)^{\otimes n}\big)^{\oplus\frac{(2n)!}{2^n.n!}}$$
où l'action de $O_{\infty,\infty}(A)$ sur le membre de droite est triviale et $M_*= {\rm Tor}_*^{\mathbf{P}(\mathbb{Z})}({\rm I}_\mathbb{Z}^\vee,{\rm I}_\mathbb{Z})$ est le groupe abélien gradué dont la description est rappelée au théorème~\ref{th-fp}.
\end{itemize}

Un résultat analogue vaut en remplaçant groupes orthogonaux hyperboliques par groupes symplectiques.
\end{cor}

\begin{cor}\label{corcor}
 Soit $k$ un corps de caractéristique nulle muni d'une anti-involution. Pour tout foncteur polynomial $F\in\F(\mathbb{Q})$ nul en zéro, le morphisme canonique
$$\underset{i\in\mathbb{N}}{\col}H_0(U_{i,i}(k);\tilde{F}(k^{2i}))\underset{\mathbb{Q}}{\otimes} H_*(U_{\infty,\infty}(k);\mathbb{Q})\to\underset{i\in\mathbb{N}}{\col}H_*(U_{i,i}(k);\tilde{F}(k^{2i}))$$
est un isomorphisme, où $\tilde{F}$ est le prolongement de $F$ aux $\mathbb{Q}$-espaces vectoriels commutant aux colimites filtrantes.
\end{cor}

\begin{proof}
 Il est bien connu (cf. par exemple \cite{MacD}, chap.~I, app.~A) qu'un foncteur polynomial sur les $\mathbb{Q}$-espaces vectoriels de dimension fini admet une décomposition en somme directe de foncteurs homogènes de degré $d$ ($d$ parcourant $\mathbb{N}$), et qu'un foncteur homogène de degré $d$ est isomorphe à un foncteur du type $V\mapsto M\underset{\Sigma_d}{\otimes}T^d(V)$, c'est donc un facteur direct de $T^d$ ; il en est donc de même pour le prolongement $\tilde{F}$ de $F$ à tous les $\mathbb{Q}$-espaces vectoriels. Le corollaire~\ref{cor-tens} (dans lequel on prend pour $A$ le foncteur d'inclusion $\mathbf{P}(k)\to\mathbf{Ab}$), combiné au théorème~\ref{th-fp} et au corollaire~\ref{cor-chpz}, montre donc que ${\rm Tor}^{\mathbf{P}(k)}_*(\mathbb{Z}[SD^2_\epsilon],F)$ est concentré en degré nul. La proposition~\ref{pr-ksh1} et le théorème~\ref{th-dvg} donnent alors le résultat annoncé (l'hypothèse $F(0)=0$ assure que tous les groupes abéliens qu'on rencontre sont des $\mathbb{Q}$-espaces vectoriels).
\end{proof}

\begin{rem}
 Si $k$ est une extension finie (ou algébrique) de $\mathbb{Q}$, on peut remplacer $\tilde{F}$ par un objet quelconque de $\F(k)$ nul en $0$, en utilisant la même démonstration et le fait que les $\mathbb{Q}$-algèbres $T^n_\mathbb{Q}(k)$ sont semi-simples (un foncteur polynomial de degré $n$ de $\F(k)$ est isomorphe à un foncteur du type  $V\mapsto M\underset{\Sigma_n}{\otimes}T^n_\mathbb{Q}(V)$ où $M$ est un $T^n_\mathbb{Q}(k)[\Sigma_n]$-module, $T^n_\mathbb{Q}(k)[\Sigma_n]$ désignant l'algèbre {\em tordue} du groupe symétrique sur la $\Sigma_n$-algèbre $T^n_\mathbb{Q}(k)$, algèbre tordue qui est elle-même semi-simple).
\end{rem}

Dans le même ordre d'idées, on dispose du résultat suivant :

\begin{cor}\label{cor-repa}
 Soient $k$ un corps commutatif de caractéristique nulle ; pour $i\in\mathbb{N}$, notons $\mathfrak{o}_{i,i}(k)$ la représentation adjointe de $O_{i,i}(k)$. Alors il existe des isomorphismes de $k$-espaces vectoriels
$$\underset{i\in\mathbb{N}}{\col}H_n(O_{i,i}(k);\mathfrak{o}_{i,i}(k))\simeq\bigoplus_{p+2q+1=n}H_p(O_{\infty,\infty}(k);\mathbb{Q})\underset{\mathbb{Q}}{\otimes}\Omega^{2q+1}_k$$
où $\Omega^*_k$ désigne la $k$-algèbre graduée $\Lambda^*_k(\Omega^1_k)$ des différentielles de Kähler de $k$ (vu comme $\mathbb{Q}$-algèbre). En particulier, l'espace vectoriel gradué $\underset{i\in\mathbb{N}}{\col}H_*(O_{i,i}(k);\mathfrak{o}_{i,i}(k))$ est nul si $k$ est une extension algébrique de $\mathbb{Q}$.
\end{cor}

\begin{proof}
 Comme $\mathfrak{o}_{n,n}(k)$ est isomorphe à $\Lambda^2_k(k^{2n})$ comme $O_{n,n}(k)$-module, il suffit de démontrer que ${\rm Tor}^{\mathbf{P}(k)}_*(\mathbb{Z}[S^2_k]^\vee,\Lambda^2_k)$ est isomorphe à la partie impaire de $\Omega^*_k$.

L'isomorphisme ${\rm Tor}^{\mathbf{P}(k)}_*(\mathbb{Z}[T^2_k]^\vee,T^2_\mathbb{Z})\simeq k[\Sigma_2]$ (concentré en degré nul) implique que, pour tout $k$-bimodule $M$, on dispose d'un isomorphisme gradué
$${\rm Tor}^{\mathbf{P}(k)}_*(\mathbb{Z}[T^2_k]^\vee,M\underset{k\otimes k}{\otimes}T^2_\mathbb{Z})\simeq HH_*(k;M)[\Sigma_2]$$
d'où l'on déduit pour $M=k$, par le théorème de Hochschild-Kostant-Rosenberg, un isomorphisme
$${\rm Tor}^{\mathbf{P}(k)}_*(\mathbb{Z}[T^2_k]^\vee,T_k^2)\simeq HH_*(k;k)[\Sigma_2]\simeq\Omega^*_k[\Sigma_2]$$
puis
$${\rm Tor}^{\mathbf{P}(k)}_*(\mathbb{Z}[T^2_k]^\vee,\Lambda_k^2)\simeq\Omega^*_k$$
où l'action de l'involution canonique est donnée par $(-1)^{n+1}$ en degré $n$. On conclut en observant (cf. démonstration du corollaire~\ref{cor-tens}) qu'on dispose d'une décomposition
$${\rm Tor}^{\mathbf{P}(k)}_*(\mathbb{Z}[T^2_k]^\vee,\Lambda_k^2)\simeq {\rm Tor}^{\mathbf{P}(k)}_*(\mathbb{Z}[S^2_k]^\vee,\Lambda_k^2)\oplus {\rm Tor}^{\mathbf{P}(k)}_*(\mathbb{Z}[\Lambda^2_k]^\vee,\Lambda_k^2)$$
dans laquelle le premier (resp. second) facteur s'identifie à la partie paire (resp. impaire) du membre de gauche pour l'action de l'involution canonique.
\end{proof}

\begin{rem}
 Une variante de cette démonstration consiste à raisonner dans la catégorie des foncteurs de  $\mathbf{P}(k)$ {\em vers les $k$-espaces vectoriels} et de noter que la trivialité de ${\rm Tor}^{\mathbf{P}(\mathbb{Q})}_*({\rm I}_\mathbb{Q}^\vee,{\rm I}_\mathbb{Q})$ implique formellement que ${\rm Tor}_*(\tilde{{\rm I}}_k^\vee,\tilde{{\rm I}}_k)$ (où $\tilde{{\rm I}}_k$ désigne le foncteur d'inclusion de $\mathbf{P}(k)$ dans les $k$-espaces vectoriels) est isomorphe à $HH_*(k;k)$.
\end{rem}

\section{Homologie stable des groupes unitaires dans quelques cas non génériques}\label{sng}

Comme dans la section précédente, on se donne une petite catégorie additive avec dualité $(\A,D)$ ; on fixe $\epsilon\in\{1,-1\}$ et un foncteur $S\in\{SD^2_\epsilon,\bar{S} D^2_\epsilon\}$. On appellera ici forme quadratique sur un objet $V$ de $\A$ tout élément de $S(V)$. Supposons que $\mathfrak{Q}$ est une classe d'espaces quadratiques {\em non dégénérés}, stable par somme orthogonale ; on note $T_\mathfrak{Q}$, ou simplement $T$, le sous-foncteur en monoïdes de  $S$ tel que $T(V)$ soit l'ensemble des formes $x$ sur $V$ telles qu'existe un morphisme de $(V,x)$ vers un élément de $\mathfrak{Q}$. Autrement dit, $T$ est le sous-foncteur en monoïdes engendré par $\mathfrak{Q}$.

Rappelons une formulation de la proposition~\ref{pr-ksh1} :
\begin{pr}\label{pr-is6} Il existe un isomorphisme de groupes abéliens gradués
$$\begin{array}{ccc}\underset{E\in\mathfrak{Q}}{\col}H_*({\rm Aut}\,(E);F(E)) & \simeq & \underset{E\in\mathfrak{Q}}{\col} H_*(\mathbf{H}^{T_\mathfrak{Q}}(\A)\times{\rm Aut}\,(E);F)\\
   & \simeq  & H_*(\mathbf{H}^{T_\mathfrak{Q}}(\A)\times U^\mathfrak{Q}_\infty(\A);F)
  \end{array}
$$
naturel en $F\in {\rm Ob}\,\mathbf{H}^{T_\mathfrak{Q}}(\A)-\mathbf{Mod}$.
\end{pr}
La colimite est prise conformément à la section~\ref{sec1} et l'on a noté $U^\mathfrak{Q}_\infty(\A)$ pour $U^{T_\mathfrak{Q}}_\infty(\A)$ --- notation introduite dans la proposition~\ref{pr-ksh1}. On utilise par ailleurs que la classe $\mathfrak{Q}$ est cofinale dans $\mathbf{H}^{T_\mathfrak{Q}}(\A)$, de sorte que le membre de gauche s'identifie à l'homologie stable notée $H^{st}(\mathbf{H}^{T_\mathfrak{Q}}(\A);F)$ en début d'article.

\smallskip

Le lemme élémentaire suivant (qui procède des mêmes arguments classiques sur les espaces hyperboliques que ceux manipulés au §\,\ref{par-dvg}) montre que la symétrisation du foncteur en monoïdes $T$ est forcément $S$ tout entier.

\begin{lm}\label{sfct-grab} 
 Tout sous-foncteur $F: \A^{op}\to\mathbf{Ab}$ de $SD^2_\epsilon$ (resp. $\bar{S}D^2_\epsilon$) engendré par des formes non dégénérées est égal à $SD^2_\epsilon$ (resp. $\bar{S}D^2_\epsilon$).
\end{lm}

\begin{proof}
 On traite d'abord le cas de $\bar{S}D^2_\epsilon$. Le résultat découle alors essentiellement du lemme~\ref{lm-scf} ; voici un argument direct : il suffit de montrer que si $x=a+\epsilon\bar{a}\in\A(V,DV)$ est une forme non dégénérée sur $V$ et $y=r+\epsilon\bar{r}$ une forme quelconque, alors $(V,y)$ possède un morphisme vers $(V,x)\overset{\perp}{\oplus} (V,-x)$ (la forme $x\oplus (-x)$ sur $V\oplus V$ appartient à $F(V\oplus V)$ si $x\in F(V)$, puisque $F$ est un foncteur vers les groupes abéliens). Posons $g=x^{-1}(r-a)$ et $f=g+1_V : V\to V$ : on vérifie que le morphisme $V\to V\oplus V$ de $\A$ de composantes $f$ et $g$ définit un morphisme hermitien $(V,y)\to (V,x)\overset{\perp}{\oplus} (V,-x)$. Cela termine la démonstration dans le cas de $\bar{S}^2_\epsilon$.

Déduisons-en maintenant le cas de $SD^2_\epsilon$. Il s'agit de montrer qu'un sous-foncteur $F$ dont la projection dans $\bar{S}D^2_\epsilon$ est $\bar{S}D^2_\epsilon$ entier égale $SD^2_\epsilon$. Soit $V$ un objet de $\A$. L'hypothèse implique l'existence d'un morphisme
$$\Phi=\left(\begin{array}{cc} f & a\\                                                                                                                                                                                                                                                                                   b & g                                                                                                                                                                                                                                                                                  \end{array}\right) : V\oplus DV\to DV\oplus V(\simeq D(V\oplus DV))$$
tel que $\Phi+\epsilon\bar{\Phi}=\left(\begin{array}{cc} 0 & 1\\                                                                                                                                                                                                                                                                                   \epsilon & 0                                                                                                                                                                                                                                                                                  \end{array}\right)$, i.e. $\bar{f}+\epsilon f=
0$, $\bar{g}+\epsilon g=
0$ et $a+\epsilon\bar{b}=1$. Le sous-groupe abélien $F(V)$ de $SD^2_\epsilon(V)=\A(V,DV)/Im\,N$, où $N$ est l'endomorphisme de $t\mapsto\bar{t}-\epsilon t$ de $\A(V,DV)$ contient, pour tous morphismes $\varphi : V\to V$ et $\psi : V\to DV$ de $\A$, la classe de
$$\left(\begin{array}{cc} \bar{\varphi} & \bar{\psi}                                                                                                                                                                                                                                                                                                                                                                                                                                                                                                                                                                    \end{array}\right)\Phi\left(\begin{array}{c} \varphi\\                                                                                                                                                                                                                                                                                   \psi                                                                                                                                                                                                                                                                                  \end{array}\right)=\bar{\varphi}f\varphi+\bar{\psi}g\psi+\bar{\varphi}a\psi+\bar{\psi}b\varphi : V\to DV.$$
En tenant compte de la relation $a+\epsilon\bar{b}=1$, on voit que $\bar{\varphi}a\psi+\bar{\psi}b\varphi-\bar{\varphi}\psi\in N$. Autrement dit, la classe de $\bar{\varphi}\psi+\bar{\varphi}f\varphi+\bar{\psi}g\psi$ appartient à $F(V)$. On en déduit, en prenant $\varphi$ nul et $\psi$ arbitraire, puis l'inverse, que la classe de $\bar{\varphi}\psi$ appartient à $F(V)$. Prenant $\varphi=1_V$, on en déduit bien $F(V)=SD^2_\epsilon(V)$ comme souhaité.
\end{proof}

Le théorème~\ref{th-sym} implique donc :
\begin{pr}\label{pr-cpts}
L'inclusion $T\hookrightarrow S$ induit un isomorphisme 
$${\rm Tor}^\A_*(\mathbb{Z}[T],F)\xrightarrow{\simeq}{\rm Tor}^\A_*(\mathbb{Z}[S],F)$$
en homologie pour $F\in {\rm Ob}\,\A-\mathbf{Mod}$ analytique.
\end{pr}

Par conséquent, si l'inclusion $\mathbf{H}^T(\A)\hookrightarrow\A^T$ induit un isomorphisme en homologie à coefficients dans un foncteur polynomial (défini sur $\A$), l'homologie stable tordue par des foncteurs polynomiaux pour les groupes unitaires associés à $\mathfrak{Q}$ se calcule à partir de l'homologie à coefficients constants (qui n'a évidemment aucune raison d'être isomorphe à celle des groupes unitaires hyperboliques) comme dans le cas générique étudié dans la section précédente. C'est cette question sur laquelle nous nous penchons maintenant.

\subsection{Le cas semi-simple}\label{par-sms}

C'est la seule situation où l'on peut obtenir sans peine une réponse positive générale à la question :

\begin{pr}\label{pr-css}
 Supposons que $\A$ est abélienne semi-simple et finie\,\footnote{Hypothèse dont on peut s'affranchir si l'on suppose que $\mathfrak{Q}$ est stable par colimites filtrantes.} (tous les objets sont de longueur finie). Alors on dispose d'un isomorphisme naturel
$$\underset{E\in\mathfrak{Q}}{\col}H_*({\rm Aut}\,(E);F(E))\simeq {\rm Tor}^{\A\times U_\infty^\mathfrak{Q}(\A)}_*(\mathbb{Z}[S],F)$$
pour $F\in {\rm Ob}\,\A-\mathbf{Mod}$ analytique.
\end{pr}

\begin{proof}
Notons $\A'$ la sous-catégorie pleine de $\A$ formée des objets $V$ tels qu'existe un {\em monomorphisme} de $(V,0)$ dans un objet de $\mathfrak{Q}$. Cette catégorie est stable par somme directe (car il en est de même pour $\mathfrak{Q}$) et par sous-objet, c'est donc une sous-catégorie abélienne de $\A$, et l'hypothèse de finitude entraîne que la sous-catégorie $\A''$ des $W$ de $\A$ tels que $\A(W,V)=0$ pour tout $V\in {\rm Ob}\,\A'$ est également abélienne et que le foncteur canonique $\A'\times\A''\to\A$ est une équivalence (tout objet $A$ de $\A$ possède un plus grand sous-objet $A'$ dans $\A'$, et $A/A'$ appartient nécessairement à $\A''$). 

Les catégories $\A^S$ et $\A^T$ se décomposent en conséquence, de sorte qu'on se ramène à traiter les deux cas particuliers suivants :
\begin{enumerate}
 \item $\A'=\A$ : on vérifie alors sans peine que $\A^T$ est nécessairement égale à $\A^S$, de sorte que l'énoncé est un cas particulier du théorème~\ref{th-dvg} ;
\item $\A'=0$ : on note d'abord qu'aucun des corps d'endomorphismes d'objets simples de $\A$ ne peut être de caractéristique finie (car sinon la sous-catégorie additive engendrée par un tel simple serait $\mathbb{F}_p$-linéaire pour un nombre premier $p$, de sorte que tout sous-foncteur en monoïdes de $S$ serait égal à $S$ par le lemme~\ref{sfct-grab}). En particulier, $2$ est inversible dans $\A$, ce qui implique que $SD^2_\epsilon$ est isomorphe à $\bar{S}D^2_\epsilon$. Par conséquent, dans tous les cas, tout objet de $\A^S$ est isomorphe à la somme orthogonale d'un objet du type $(V,0)$ ($V$ est le noyau de la forme) et d'un objet non dégénéré.

Cela entraîne que le foncteur d'inclusion $\mathbf{H}^T(\A)\to\A^T$ possède alors un adjoint à gauche associant à un espace quadratique son quotient par son radical --- le point à noter est que si $f : E \to F$ est un morphisme de $\A^T$, $F({\rm Rad}\,E)\subset {\rm Rad}\,F$, car la composée ${\rm Rad}\,E\hookrightarrow E\xrightarrow{f} F\twoheadrightarrow F/{\rm Rad}\,F\to Q$, où l'on a choisi une flèche $F/{\rm Rad}\,F\to Q$ avec $Q\in\mathfrak{Q}$ (ce qui est possible puisque $F/{\rm Rad}\,F$ est isomorphe à un sous-espace de $F$ --- cf. début de la démonstration), est forcément nulle (son image a une forme nulle et s'envoie injectivement dans $Q$), ce qui implique le résultat puisque $F/{\rm Rad}\,F\to Q$ est injective (sa source est non dégénérée).
\end{enumerate}
\end{proof}

On en déduit par exemple :
\begin{cor}\label{cor-hil3}
 Soient $k$ un corps commutatif de caractéristique nulle et $n\in\mathbb{N}^*$. On a
$$\underset{i\in\mathbb{N}}{\col}H_*(O_i(k);\Lambda^n_\mathbb{Q}(k^i))=0.$$

(Les groupes $O_i$ qui apparaissent sont ici les groupes~\guillemotleft~euclidiens~\guillemotright~associés aux formes $\sum_{j=1}^i X_j^2$.)
\end{cor}

\begin{proof}
 Raisonnant à partir de la proposition précédente comme pour démontrer le corollaire~\ref{corcor}, on voit qu'il suffit d'établir cette assertion en degré homologique nul. Elle provient alors de ce que les coïnvariants de $\Lambda^n_\mathbb{Q}(k^i)$ sous l'action du sous-groupe $\{\pm 1\}^i$ des matrices diagonales de $O_i(k)$ sont déjà clairement nuls.
\end{proof}

Dans le cas où $k$ est le corps des nombres réels, cette propriété est une forme affaiblie d'un résultat de Dupont et Sah (\cite{DS}, théorème~2.2 ; voir aussi \cite{D-hilb}, chap.~9, où d'autres problèmes analogues d'homologie des~\guillemotleft~groupes de Lie rendus discrets~\guillemotright~sont présentés), qui stipule que $H_j(O_i(\mathbb{R});\Lambda^n_\mathbb{Q}(\mathbb{R}^i))$ est nul dès que $n\geq i+j$. Noter qu'il semble fort improbable qu'un tel résultat (qui inclut une borne de stabilité forte) soit valable pour tous les corps de caractéristique nulle (d'ailleurs, la démonstration de Dupont-Sah utilise de façon importante que les sommes de carrés de nombres réels sont encore des carrés), puisque tous les résultats de stabilité homologique (même à coefficients constants) sur les groupes orthogonaux~\guillemotleft~euclidiens~\guillemotright~nécessitent des hypothèses arithmétiques sur le corps (cf. \cite{Vog} et \cite{Col}).

Rappelons que la motivation majeure du résultat d'annulation de Dupont-Sah mentionné est le problème des groupes de découpage de polyèdres affines euclidiens (troisième problème de Hilbert) : celui-là constitue une étape cruciale dans la démonstration que volume et invariant de Dehn caractérisent les polyèdres d'un espace affine euclidien (réel) de dimension~$3$ à découpage près (cf. \cite{D-hilb}, où l'on trouvera aussi de nombreuses autres références sur ce riche sujet).

Citons un autre résultat, variante du corollaire~\ref{cor-repa} :

\begin{cor}\label{cor-repa2}
 Soient $k$ un corps commutatif de caractéristique nulle ; pour $n\in\mathbb{N}$, notons $\mathfrak{o}_{n}(k)$ la représentation adjointe de $O_{n}(k)$. Alors le $k$-espace vectoriel gradué
$$\underset{n\in\mathbb{N}}{\col}H_*(O_{n}(k);\mathfrak{o}_{n}(k))$$
est isomorphe au produit tensoriel (sur $\mathbb{Q}$) de $H_*(O_{\infty}(k);\mathbb{Q})$ et de la partie impaire de la $k$-algèbre graduée $\Omega^*_k=\Lambda^*(\Omega^1_k)$ des différentielles de Kähler de $k$ (vu comme $\mathbb{Q}$-algèbre). En particulier, 
$$\underset{n\in\mathbb{N}}{\col}H_1(O_{n}(k);\mathfrak{o}_{n}(k))\simeq\Omega^1_k\quad\text{et}\quad\underset{n\in\mathbb{N}}{\col}H_2(O_{n}(k);\mathfrak{o}_{n}(k))=0.$$
\end{cor}
(Le groupe orthogonal n'a pas d'homologie rationnelle de degré $1$ puisqu'il est engendré par des symétries.)

Pour $k=\mathbb{R}$, la dernière partie de ce corollaire est encore une version affaiblie (stable) d'un résultat (valable en degré homologique $n>4$) dû à Dupont (\cite{Dua}, cor.~4.18), dont la formulation en termes de représentation adjointe et de différentielles de Kähler (avec une description explicite de l'isomorphisme) est due à Cathelineau (\cite{Ca}).

\begin{rem}
 \begin{enumerate}
  \item Si $k$ est un corps commutatif, notons $MW(k)$ le monoïde ordonné des classes d'équivalence de $k$-espaces quadratiques non dégénérés sur $k$ (ainsi, le groupe de Grothendieck-Witt de $k$ est la symétrisation de $MW(k)$). On se convainc aisément que l'ensemble des sous-foncteurs en monoïdes du foncteur des formes quadratiques sur les $k$-espaces vectoriels de dimension finie engendrés par des espaces non dégénérés est en bijection avec les sous-monoïdes $M$ de $MW(k)$ tels que $x\leq t$ avec $t\in M$ implique $x\in M$. On voit aussitôt que de tels sous-monoïdes sont caractérisés par leur image dans le quotient du groupe de Witt de $k$ par son sous-groupe de torsion, lequel est directement contrôlé par les relations d'ordre total compatibles à la structure de corps de $k$ (cf. \cite{Lam}, chap.~XIII, par exemple). 

En particulier, si $k$ n'est pas ordonnable, il n'y a pas de situation~\guillemotleft~non générique~\guillemotright~; si $k$ possède un et un seul ordre total compatible avec sa structure de corps (par exemple, si $k$ est un corps ordonné maximal ou $k=\mathbb{Q}$), le seul groupe orthogonal stable non générique qu'on peut obtenir avec le formalisme ici employé est le groupe~\guillemotleft~euclidien~\guillemotright~infini $\underset{n\in\mathbb{N}}{\col}O_n(k)$. 
  \item Il s'impose de déterminer quels analogues des résultats de ce paragraphe subsistent pour les groupes orthogonaux de type $O_{n,i}$ (où $n$ parcourt $\mathbb{N}$ et $i\in\mathbb{N}^*$ est fixé) sur un corps : ceux-ci n'entrent pas exactement dans le cadre présenté ici, et sont l'objet d'un travail en cours de l'auteur. L'existence de résultats de comparaison partiels (pour le corps des réels et $i=1$, comparé au cas euclidien) à coefficients constants, dus à Bökstedt, Brun et Dupont (\cite{BBD}) rend particulièrement naturelle cette question.
 \end{enumerate}
\end{rem}

\subsection{Cas des coïnvariants pour les groupes $O_n$}\label{par-h0}

Désormais, on se restreint au cas particulier suivant : {\em $A$ est un anneau commutatif où $2$ est inversible, $\A$ est la catégorie des $A$-modules projectifs (ou libres\,\footnote{Les deux catégories étant équivalentes au sens de Morita, on peut travailler indifféremment avec l'une ou l'autre.}) de rang fini, munie de la dualité usuelle, $\epsilon=1$.} On traite donc d'espaces quadratiques et de groupes orthogonaux au sens ordinaire sur les $A$-modules projectifs (ou libres).

Pour considérer les groupes de type $O_n(A)$, on prend pour $\mathfrak{Q}$ la classe des espaces quadratiques $(A^n,\varepsilon_n)$ (où  $\varepsilon_n$ est la forme $X_1^2+\dots+X^2_n$). Noter que $T(V)$ est le monoïde (additif) des formes quadratiques qui sont somme de carrés de formes linéaires sur $V$.

\begin{nota}
 On désigne par $rq(A)$ l'ensemble des éléments $x$ de $A$ tels que $1-x^2$ soit somme de carrés dans $A$.
\end{nota}

(Les lettres $rq$ sont mises pour {\em radical quadratique}.)

\begin{pr}\label{prs-rq}
 \begin{enumerate}
  \item Un élément de $A$ appartient à $rq(A)$ si et seulement s'il est coefficient d'une matrice de $O_n(A)$ pour un entier $n>0$.
\item Le sous-ensemble $rq(A)$ est stable par opposé et par produit, il contient $1$.
 \end{enumerate}
\end{pr}

\begin{proof}
 Un coefficient d'une matrice orthogonale appartient clairement à $rq(A)$. Réciproquement, si $\sum_{i=1}^n a_i^2=1$, il existe un élément de $O_{n+1}(A)$ dont la première colonne soit $(0,a_1,\dots,a_n)$, par exemple la réflexion orthogonale de $(A^{n+1},\varepsilon_{n+1})$ par rapport à l'hyperplan orthogonal à $(1,-a_1,\dots,-a_n)$.

Par conséquent, si $x$ et $y$ sont des éléments de $rq(A)$, pour $n$ assez grand, on peut trouver $A$ et $B$ dans $O_n(A)$ dont les premières colonnes sont respectivement du type $(0,t_2,\dots,t_i,0,\dots,0,x)$ et $(0,\dots,0,u_1,\dots,u_j,y)$ (avec $i$ zéros au début), de sorte que le premier coefficient de la matrice orthogonale $^t A B$ est $xy$, qui est donc dans $rq(A)$. Le reste est évident.
\end{proof}

\begin{cor}\label{cor-rq}
 Le sous-ensemble de $A$ des sommes d'éléments de $rq(A)$ en est un sous-anneau, noté ${\rm Rq}(A)$. C'est le plus petit sous-anneau de $A$ tel que $O_n({\rm Rq}(A))=O_n(A)$ pour tout $n$. 
\end{cor}

Ainsi, Rq définit un endofoncteur involutif de la catégorie des anneaux (commutatifs avec $2$ inversible).

Une manière de trouver des éléments dans $rq(A)$ est d'appliquer la proposition suivante :
\begin{pr}\label{pr-scs}
 Soient $n>2$ un entier et $(a_i)_{1\leq i\leq n}$ une famille d'éléments de $A$ telle que $\sum_i a_i^2\in A^\times$. Alors $\frac{2a_1 a_2}{\sum_i{a^2_i}}\in rq(A)$.
\end{pr}

\begin{proof}
 Notons $s=\sum_{i>2}a^2_i$. Alors
$$\big(\sum_i{a^2_i}\big)^2-(2a_1a_2)^2=(a_1^2+a_2^2+s)^2-4a_1^2a_2^2=(a_1-a_2)^2+s^2+2(a_1^2+a_2^2)s$$
est une somme de carrés, d'où la proposition.
\end{proof}

\begin{lm}\label{lm-sc}
\begin{enumerate}
 \item  Si $x\in {\rm Rq}(A)$, alors il existe un entier $n\in\mathbb{N}$ tel que $n-x^2$ soit une somme de carrés. 
\item Si $t$ est une somme de carrés d'éléments de ${\rm Rq}(A)$, alors il existe une somme de carrés $s$ telle que $s+t\in A^\times$.
\end{enumerate}
\end{lm}

\begin{proof}
 Notons $E$ le sous-ensemble de $A$ formé des $x$ tel qu'existe $n\in\mathbb{N}$ tel que $n-x^2$ soit une somme de carrés. Comme $E$ contient $rq(A)$, il suffit de montrer que $E$ est stable par somme. Soient en effet $x$, $y$ des éléments de $E$ et $n, m\in\mathbb{N}$ tels que $n-x^2$ et $m-y^2$ soient des sommes de carrés. Alors
$$2(n+m)-(x+y)^2=2\big((n-x^2)+(m-y^2)\big)+(x-y)^2$$
est une somme de carrés, d'où le premier point.

Le deuxième point en découle dans la mesure où $2$ est inversible dans $A$ : on peut trouver une somme de carrés $s$ telle que $s+t$ soit un entier ; quitte à ajouter suffisamment de termes $1$ à $s$, on peut faire en sorte que cet entier soit une puissance de~$2$.
\end{proof}

\begin{thm}\label{th-h0gl}
 Si ${\rm Rq}(A)=A$, alors, pour tout foncteur analytique $F\in {\rm Ob}\,\mathbf{P}(A)-\mathbf{Mod}$, le morphisme canonique
$$\underset{n\in\mathbb{N}}{\col}F(A^n)_{O_n(A)}\simeq H_0(\mathbf{H}^T(\A);i^*\pi^*F)\to H_0(\A^T;\pi^*F)\simeq\mathbb{Z}[S^2_A]^\vee\underset{\mathbf{P}(A)}{\otimes} F$$
induit par l'inclusion $i : \mathbf{H}^T(\A)\to\A^T$ est un isomorphisme. (On rappelle que $(-)^\vee$ désigne la précomposition par la dualité.)
\end{thm}

\begin{proof}
 On applique le corollaire~\ref{cor-auxh0} en montrant l'assertion suivante, par récurrence sur l'entier $n$ :

{\em pour tout objet $(V,q)$ de $\A^T$ tel que $V\simeq A^n$ (comme $A$-module), il existe $Q\in T(V)$ tel que $q+Q$ soit non dégénérée.}

(Cela suffit puisque, par définition de $T$, tout objet de $\A^T$ possède un morphisme vers un objet de $\mathbf{H}^T(\A)$.) 

Pour $n=0$ il n'y a rien à démontrer ; on suppose donc $n>0$ et l'assertion établie pour les entiers strictement inférieurs. Soient $(V,q)$ un objet de $\A^T$ avec $V$ de rang $n$, $(e_1,\dots,e_n)$ une base de $V$ et $l\in V^*$ tel que $l(e_1)=1$. Comme $q\in T(V)$, par le lemme~\ref{lm-sc}, on peut trouver une somme de carrés $s$ telle que $q(e_1)+s\in A^\times$. Posons $Q'=s.l^2$ --- c'est un élément de $T(V)$ --- et $q'=q+Q'$.

Du fait que $q'(e_1)\in A^\times$, $V$ est somme directe de la droite qu'engendre $e_1$ et de son orthogonal $W$, qui est isomorphe à $A^{n-1}$ (l'application linéaire $\bigoplus_{i>1} Ae_i\to W\quad v\mapsto v-\frac{B(v,e_1)}{B(e_1,e_1)}e_1$, où $B$ est la forme bilinéaire associée à $q'$, est bijective). Par l'hypothèse de récurrence, on peut trouver une forme $Q''\in T(W)$ telle que $q'|_W+Q''$ soit non dégénérée. Comme $(V,q')\simeq (A,q'(e_1)\varepsilon_1)\overset{\perp}{\oplus}(W,q'|_W)$, la forme $Q''$ fournit une forme $\bar{Q}$ sur $V$ appartenant à $T(V)$ telle que $q'+\bar{Q}$ soit non dégénérée. Par conséquent, la forme $Q=Q'+\bar{Q}\in T(V)$ convient, d'où le théorème.
\end{proof}

Tenant compte du corollaire~\ref{cor-rq}, on en déduit :

\begin{cor}\label{corh0}
 Pour tout foncteur analytique $F\in {\rm Ob}\,\mathbf{P}(A)-\mathbf{Mod}$, il existe un isomorphisme naturel
$$\underset{n\in\mathbb{N}}{\col}F(A^n)_{O_n(A)}\simeq\mathbb{Z}[S^2_{{\rm Rq}(A)}]^\vee\underset{\mathbf{P}({\rm Rq}(A))}{\otimes} F(A\underset{{\rm Rq}(A)}{\otimes}-)$$
\end{cor}

\begin{nota}
 On note $\widetilde{{\rm Rq}}(A)$ le sous-anneau de $A$ engendré par ${\rm Rq}(A)$ et les éléments $x^{-1}$, où $x\in {\rm Rq}(A)$ est inversible dans $A$.
\end{nota}

Par platitude de $\widetilde{{\rm Rq}}(A)$ sur ${\rm Rq}(A)$, on en déduit (par la proposition~\ref{pr-chplat}) :

\begin{cor}\label{corh0-t}
 Pour tout foncteur analytique $F\in {\rm Ob}\,\mathbf{P}(A)-\mathbf{Mod}$, il existe un isomorphisme naturel
$$\underset{n\in\mathbb{N}}{\col}F(A^n)_{O_n(A)}\simeq\mathbb{Z}[S^2_{\widetilde{{\rm Rq}}(A)}]^\vee\underset{\mathbf{P}(\widetilde{{\rm Rq}}(A))}{\otimes} F(A\underset{\widetilde{{\rm Rq}}(A)}{\otimes}-)$$
\end{cor}

\begin{ex}
\begin{enumerate}
\item On tire des corollaires~\ref{corh0} et~\ref{cor-tens} un isomorphisme canonique
$$\underset{n\in\mathbb{N}}{\col}S^2_\mathbb{Z}(A^n)_{O_n(A)}\simeq S^2_{{\rm Rq}(A)}(A).$$
 \item Supposons que $k$ est un corps ordonné et que $A=k[X]$ est un anneau de polynômes sur $k$. Alors ${\rm Rq}(A)=k$. (Si $\sum_i P^2_i=1$ dans $A$, considérer le degré maximal $d$ des $P_i$ : le coefficient de $X^{2d}$ dans la somme est une somme de carrés de $k$, qui n'est pas nulle puisque $k$ est ordonné, ce qui impose $d=0$.) On en déduit que le morphisme canonique
$$\underset{n\in\mathbb{N}}{\col}S^2_\mathbb{Z}(A^n)_{O_n(A)}\to\underset{n\in\mathbb{N}}{\col}S^2_\mathbb{Z}(A^{2n})_{O_{n,n}(A)}$$
s'identifie au produit
$$S^2_k(k[X])\to k[X],$$
qui {\em n'est pas} un isomorphisme.

Ainsi, la proposition~\ref{pr-css} ne se généralise pas sans restriction, même pour les modules projectifs de type fini sur un anneau principal !
\item Si l'anneau $A$ est local, alors $\widetilde{{\rm Rq}}(A)=A$. En effet, soit $a$ un élément de $A$ tel que $1+a^2\in A^\times$ : la somme de carrés $1+1+a^2+a^2$ est inversible, donc la proposition~\ref{prs-rq} montre que $\frac{a}{1+a^2}$ et $\frac{1}{1+a^2}$ sont dans $rq(A)$. Par conséquent, $a\in\widetilde{{\rm Rq}}(A)$. Comme il n'est pas possible que $1+a^2$, $1+(1+a)^2$ et $1+(1-a)^2$ soient simultanément non inversibles (projeter dans le corps résiduel), on en déduit notre assertion.
\item En revanche, même si $A$ est un corps, on n'a pas nécessairement ${\rm Rq}(A)=A$. En effet, si $k$ est un corps ordonné, on vérifie aisément l'inclusion ${\rm Rq}\big(k((X))\big)\subset k[[X]]$.
\item Il s'impose de déterminer si le théorème~\ref{th-h0gl} (et donc ses corollaires) s'étend en degré homologique supérieur. Il est patent que la méthode utilisée dans le paragraphe suivant s'applique potentiellement à d'autres cas que celui où on la met en \oe uvre, mais certainement pas à tous les anneaux. L'auteur doute d'ailleurs de la véracité d'une telle généralisation du théorème~\ref{th-h0gl} sans aucune hypothèse supplémentaire, mais trouver des contre-exemples semble très ardu.
\end{enumerate}
\end{ex}

\subsection{Anneaux d'entiers de corps de nombres}\label{par-saq}

Dans cette section, on se donne une extension finie\,\footnote{On peut étendre facilement les résultats de ce paragraphe à une extension algébrique quelconque en l'écrivant comme colimite filtrante de ses sous-extensions finies.} $K$ de $\mathbb{Q}$ et une partie multiplicative de $\mathbb{Z}$ contenant $2$ ; $A$ est le sous-anneau de $K$ obtenu en localisant relativement à cette partie multiplicative l'anneau des entiers de $K$. La classe $\mathfrak{Q}$ est toujours la même. 

On note $\Pl$ l'ensemble (fini) des morphismes de corps de $K$ dans $\mathbb{R}$ (la notation vient d'une abréviation de {\em plongements}). On supposera que $\Pl$ est non vide, sans quoi $-1$ est une somme de carrés dans $A$ et il n'y a rien à faire.

On notera $x\preceq y$, pour $x$ et $y$ dans $K$, pour signifier que $y-x$ est {\em totalement positif}, c'est-à-dire une somme de carrés ($\preceq$ est donc une relation d'ordre {\em partiel} compatible avec la structure de corps de $K$). Cela équivaut à dire que $\alpha(x)\leq\alpha(y)$ pour tout $\alpha\in\Pl$ (cf. par exemple \cite{Lam}, chap. XIII).

\begin{lm}\label{lm-scar}
 Si $a\preceq b$ avec $a$ et $b$ dans $A$, alors $b-a$ est une somme de carrés de $A$.
\end{lm}

\begin{proof}
 Posons $S=A\cap\mathbb{Q}$  : ce sous-anneau de $\mathbb{Q}$ contient $\mathbb{Z}[1/2]$, il est donc {\em dense} (pour la topologie usuelle). Alors $A$ est un $S$-module libre de type fini (puisque les entiers de $K$ forment un groupe abélien libre de type fini) ; soit $(a_i)_{1\leq i\leq d}$ une base --- c'est aussi une base de $K$ sur $\mathbb{Q}$.

Donnons-nous $x\succ 0$ dans $A$ ; écrivons
$$x=\sum_i{t_i a_i}\quad\text{et}\quad\sum_i{\Big(\frac{1}{4}-a_i\Big)^2 =\sum_i{u_i a_i}}.$$
Par densité, il existe $\epsilon>0$ dans $S$ tel que $\sum_i{\Big(t_i-\epsilon\big(\frac{1}{2}+u_i\big)\Big)a_i}\succ 0$ (utiliser la caractérisation de la relation $\succ$ à l'aide des éléments de $\Pl$, lesquels sont en nombre fini et laissent $\mathbb{Q}$ invariant). Par conséquent, on peut trouver des éléments $x_1,\dots,x_n$ de $K$ tels que $\sum_j x_j^2=x-\epsilon\sum_i a_i$. Il s'ensuit que le sous-ensemble
$$\{\lambda=(\lambda_{i,j})_{1\leq i\leq d, 1\leq j\leq n}\,|\,\forall 1\leq i\leq d\qquad t_i-\epsilon u_i-\epsilon<f_i(\lambda)<t_i-\epsilon u_i \}$$
de $\mathbb{Q}^{nd}$, où les $f_i$ sont les fonctions polynomiales caractérisées par
$$\sum_j\Big(\sum_i\lambda_{i,j}a_i\Big)^2=\sum_i f_i(\lambda)a_i\;,$$
est non vide. Comme c'est un ouvert, il contient un élément $\lambda$ de  la partie dense $S^{nd}$. Ainsi,
$$x=\sum_j\Big(\sum_i\lambda_{i,j}a_i\Big)^2 +\sum_i (t_i-f_i(\lambda)-\epsilon u_i)a_i+\epsilon\sum_i(1/4-a_i)^2$$
est une somme de carrés de $A$ puisque
$$\sum_i (t_i-f_i(\lambda)-\epsilon u_i)a_i+\epsilon\sum_i(1/4-a_i)^2=\dots$$
$$=\sum_i (t_i-f_i(\lambda)-\epsilon u_i)(1/4+a_i)^2+\sum_i (\epsilon-t_i+f_i(\lambda)+\epsilon u_i)(1/4-a_i)^2$$
et que les éléments $t_i-f_i(\lambda)-\epsilon u_i$ et $\epsilon-t_i+f_i(\lambda)+\epsilon u_i$ de $S$ sont par hypothèse positifs, donc sommes de carrés.
\end{proof}

\begin{rem}
 On peut montrer, en utilisant davantage d'arithmétique, que tout élément totalement positif de $A$ est somme de cinq carrés (dans $A$). L'auteur doit cette remarque à G. Collinet.
\end{rem}

On munit $K$ de la topologie obtenue par restriction de la topologie usuelle de $\mathbb{R}^\Pl$ au moyen de l'injection canonique.

On utilisera couramment les propriétés suivantes :
\begin{lm}\label{lm-ptop}
 \begin{enumerate}
\item La topologie de $K$ est compatible à sa structure de corps.
  \item Le sous-ensemble $A$ de $K$ est dense.
  \item Pour tout $x\in K$, les sous-ensembles $\{t\in K\,|\,t\preceq x\}$ et $\{t\in K\,|\,t\succeq x\}$ sont fermés dans $K$, tandis que $\{t\in K\,|\,t\prec x\}$ et $\{t\in K\,|\,t\succ x\}$ sont ouverts.
 \end{enumerate}
\end{lm}

\begin{proof}
 La première assertion résulte de ce que les éléments de $\Pl$ sont des morphismes de corps et que la topologie usuelle de $\mathbb{R}$ est compatible à sa structure de corps.

Le deuxième point provient de ce que $A\cap\mathbb{Q}$, qui contient $\mathbb{Z}[1/2]$, est dense dans $\mathbb{Q}$ (pour la topologie usuelle, qui est aussi celle qu'induit la topologie qu'on a choisie sur $K$) et que $K$ est engendré comme monoïde multiplicatif par $A$ et $\mathbb{Q}$, puisque $A$ contient les entiers de $K$.

La dernière assertion découle de la compatibilité de la topologie de $\mathbb{R}$ à son ordre et, pour le point concernant les inégalités strictes, de la finitude de $\Pl$. 
\end{proof}

On munit tout $K$-espace vectoriel (ou tout $K$-espace affine) de dimension finie de la topologie la moins fine rendant continues toutes les formes linéaires. Cette topologie est compatible à la structure de $K$-espace vectoriel, et toutes les applications linéaires entre $K$-espaces vectoriels de dimension finie sont continues. Si $M$ est un $A$-module projectif de type fini, on déduit du deuxième point du lemme que le sous-ensemble $M$ du $K$-espace espace vectoriel $M_K=K\underset{A}{\otimes} M$ est dense.

Par ailleurs :

\begin{lm}\label{lm-dinv}
 Étant donné $a\succ 0$ dans $A$, il existe $t\in A^\times$ tel que $a\succeq t^2$.
\end{lm}

\begin{proof}
Comme tous les réels $\alpha(a)$ sont strictement positifs et que $\Pl$ est fini, on a $\alpha(2^{-2n})=2^{-2n}\preceq\alpha(a)$ pour tout $\alpha\in\Pl$ et $n\in\mathbb{N}$ suffisamment grand. Pour un tel $n$, $t=2^{-n}$ convient. 
\end{proof}

Dans la suite, nous nous appuyerons uniquement sur les lemmes~\ref{lm-scar}, \ref{lm-ptop} et~\ref{lm-dinv}.

\begin{lm}\label{lm-prtp}
\begin{enumerate}
 \item Soient $V$ un $A$-module projectif de type fini et $q$ une forme quadratique sur $V$. Si $q(v)\succeq 0$ (resp. $q(v)\succ 0$) pour tout $v\in V\setminus 0$, alors $q_K(x)\succeq 0$ (resp. $q_K(x)\succ 0$) pour tout $x\in V_K\setminus 0$.
\item Soit $E$ un $K$-espace vectoriel de dimension finie. L'ensemble des formes quadratiques $Q$ sur $K$ telles que $Q(x)\succ 0$ pour tout $x\in E\setminus 0$ est ouvert dans l'espace vectoriel des formes quadratiques sur $E$.
\end{enumerate}
\end{lm}
\begin{proof}
Le premier point résulte de ce que $V_K$ est une localisation de $V$. 

Pour la seconde assertion, on raisonne par récurrence sur la dimension $n$ de $E$ ; le cas $n\leq 1$ est immédiat (à partir du lemme~\ref{lm-ptop}). Supposons $n>1$ ; considérons une base $(e_1,\dots,e_n)$ de $V$ et notons $W$ le sous-espace engendré par les $e_i$ pour $i>1$. La condition $Q(x)\succ 0$ pour $x\neq 0$ ($Q\succ 0$ en abrégé) équivaut à la conjonction des assertions suivantes :
\begin{enumerate}
 \item $Q(e_1)\succ 0$ ;
\item la restriction de la forme $Q(e_1)Q-B(e_1,-)^2$ (où $B$ est la forme bilinéaire associée à $Q$) à $W$ est $\succ 0$.
\end{enumerate}
(La vérification est une variation autour de l'inégalité de Cauchy-Schwartz.)

Par conséquent, l'hypothèse de récurrence et le lemme~\ref{lm-ptop} donnent la conclusion.
\end{proof}

On note $\mathbf{Q}_{++}$ la sous-catégorie pleine de $\A^T$ des espaces quadratiques $(V,q)$ tels que $q(x)\succ 0$ pour $x\in V\setminus\{0\}$.

\begin{lm}\label{lm-fqq}
 \begin{enumerate}
  \item La sous-catégorie $\mathbf{H}^T(\A)$ est incluse dans $\mathbf{Q}_{++}$ ;
\item l'inclusion $\mathbf{Q}_{++}\to\A^T$ possède un adjoint à gauche ;
\item tout $A$-espace quadratique $(V,q)$ (où $V$ est un $A$-module projectif de type fini) tel que $q(x)\succ 0$ appartient à $\mathbf{Q}_{++}$ ; plus précisément, il existe un morphisme de $(V,q)$ dans un espace quadratique $(A^n,\varepsilon_n)$ qui est un monomorphisme scindé comme morphisme de $A$-modules.
 \end{enumerate}
\end{lm}

\begin{proof}
 Si $(V,q)$ est un espace quadratique sur $A$ tel que $q(v)\succeq 0$ pour tout $v\in V$, alors $q(x)=0$ entraîne $x\in {\rm Rad}\,(V,q)$ (pour tout $y\in V$, la fonction $A\to A\quad q(tx+y)$ est affine et à valeurs totalement positives, donc constante, par densité de $A$ dans $K$). Cela fournit les deux premiers points (l'adjoint à gauche s'obtient en quotientant par le radical).

Pour la dernière assertion, il suffit de traiter le cas où $V$ est libre. On raisonne alors par récurrence sur le rang $d$ de $V$. Pour $d=1$, cela provient du lemme~\ref{lm-dinv}.

On suppose maintenant $d>1$ et l'assertion vérifiée pour les espaces de rang $d-1$. Soient $(e_1,\dots,e_d)$ une base de $V$ et $a_1,\dots,a_n$ des éléments de $A$ tels que $q(e_1)=\sum_i a_i^2$ avec $a_1$ inversible (lemme~\ref{lm-dinv}). Notons $\E$ l'ensemble des $n$-uplets $(l_1,\dots,l_n)$ de formes linéaires sur $V$ vérifiant :
\begin{enumerate}
 \item $l_i(e_1)=a_i$ pour tout $1\leq i\leq d$ ;
\item $\underset{i}{\sum} a_i l_i=B(e_1,-)$ où $B$ est la forme bilinéaire associée à $q$.
\end{enumerate}

Alors $\E$ est un sous-espace affine de $(V^*)^n$ --- il est non vide car, si $L$ est une forme linéaire telle que $L(e_1)=1$, poser $l_i=a_i L$ pour $i>1$ et $l_1=a_1^{-1}\Big(B(e_1,-)-\underset{i>1}{\sum}a_i l_i\Big)$ fournit un élément de $\E$.

Si ${\rm l}=(l_i)$ est un élément de $\E$, $e_1$ appartient au radical de la forme quadratique $q-\sum_i l_i^2$, laquelle induit donc une forme quadratique $Q_{\rm l}$ sur $W=V/e_1$.

Considérons l'ensemble des $\lambda\in\E_K$ tels que $Q_\lambda(t)\succ 0$ pour tout $t\in W_K\setminus 0$ (on rationalise la construction précédente). Il est ouvert dans $\E_K$ par les lemmes~\ref{lm-prtp} et~\ref{lm-ptop}. Il est également non vide : si l'on pose $l_i=\frac{a_i}{q(e_1)}B_K(e_1,-)$, $\lambda=(l_i)$ appartient à $\E_K$ et 
$$\sum_i l_i^2=\frac{B_K(e_1,-)^2}{q(e_1)}$$
de sorte que l'on conclut par l'inégalité de Cauchy-Schwartz. La densité de $\E$ dans $\E_K$ implique donc qu'existe ${\rm l}\in\E$ tel que ${\rm l}\in\E$ tel que $Q_{\rm l}(t)\succ 0$ pour tout $t\in W\setminus 0$.

 On conclut en appliquant l'hypothèse de récurrence à la forme quadratique $Q_{\rm l}$ (la relation $l_1(e_1)=a_1\in A^\times$ implique qu'on obtient bien un monomorphisme scindé).
\end{proof}

\begin{lm}\label{ineg-dens}
 Soient $(V,q)$, $(E,Q)$ des $A$-espaces quadratiques (avec $V$ et $E$ projectifs de type fini) et $f : V\to E$ une application linéaire telle que $q(x)\succeq Q(f(x))$ pour tout $x\in V$.

Alors il existe $n\in\mathbb{N}$ et un morphisme quadratique $(V,q)\to (E,Q)\overset{\perp}{\oplus} (A^n,\varepsilon_n)$ dont la composante $V\to E$ est $f$.
\end{lm}

\begin{proof}
 Il s'agit de trouver des formes linéaires $l_1,\dots,l_n$ sur $V$ telles que 
$$q(x)=Q(f(x))+\sum_i l_i^2(x)\;;$$
par conséquent, il suffit de traiter le cas $E=0$. Si $q$ ne s'annule qu'en $0$, c'est un cas particulier du dernier point du lemme précédent ; le cas général s'y ramène en tuant le radical (cf. début de la démonstration du lemme précédent).
\end{proof}

\begin{lm}\label{lm-ccd}
 Soit
$$\xymatrix{V\ar[r]^i\ar[rd]_u & C\\
& B
}$$
un diagramme de $\mathbf{Q}_{++}$ tel que $\pi(u)$ soit un monomorphisme scindé ; notons $\beta$ et $\gamma$ les formes bilinéaires sur $B$ et $C$ respectivement.

On suppose qu'il existe une application linéaire $f : \pi(B)\to\pi(C)$ telle que :
\begin{enumerate}
 \item $f\circ u=i$ ;
\item $\forall (a,b)\in V\times B\quad \beta(u(a),b)=\gamma(i(a),f(b))$.
\end{enumerate}

Alors il existe $n\in\mathbb{N}$ et un diagramme commutatif de $\mathbf{Q}_{++}$ du type
$$\xymatrix{V\ar[r]^i\ar[rd]_u & C\ar[r] & C\overset{\perp}{\oplus}(A^n,\varepsilon_n)\\
& B\ar[ru] &
}$$
\end{lm}

\begin{proof}
 Notons $\E$ le sous-espace affine de ${\rm Hom}_A(\pi(B),\pi(C))$ des applications linéaires $f$ vérifiant les deux conditions de l'énoncé, $q_B$ et $q_C$ les formes quadratiques respectives sur $B$ et $C$. Pour $f\in\E$, la forme quadratique $q_B-\phi^*q_C$ a un radical contenant $Im\,u$, en vertu des égalités
$$\beta(u(a),b)=\gamma(i(a),f(b))=\gamma(f(u(a)),f(b))\;;$$
par conséquent, elle induit une forme quadratique $Q_f$ sur $Coker\,u$, qui est un $A$-module projectif de type fini puisque $\pi(u)$ est un monomorphisme scindé. Le lemme~\ref{ineg-dens} montre qu'il suffit d'établir l'existence de $f\in\E$ telle que $Q_f\succeq 0$.

Commençons par le faire en tensorisant par $K$ : comme le rationalisé d'un espace de $\mathbf{Q}_{++}$ est non dégénéré, le diagramme initial se rationalise en
$$\xymatrix{V_K\ar[r]^-{i_K}\ar[rd]_-{u_K} & C_K\simeq V_K\overset{\perp}{\oplus} W\\
& B_K\simeq V_K\overset{\perp}{\oplus} U
}$$
et l'application linéaire $\phi : B_K\twoheadrightarrow V_K\hookrightarrow C_K$ est un élément de $\E_K$ tel que $Q_\phi\succ 0$ (cette forme quadratique s'identifie à la forme sur $U$). On conclut par un argument de densité identique à celui employé pour démontrer le lemme~\ref{lm-fqq}.
\end{proof}

\begin{thm}\label{th-anord}
 L'inclusion $i : \mathbf{H}^T(\A)\to\A^T$ induit un isomorphisme
$$H_*(\mathbf{H}^T(\A);i^*\pi^* F)\to H_*(\A^T;\pi^*F)\simeq {\rm Tor}^{\mathbf{P}(A)}_*(\mathbb{Z}[S^2_A]^\vee,F)$$
pour tout foncteur analytique $F\in {\rm Ob}\,\F(A)$.
\end{thm}

\begin{proof}
 On raisonne de façon assez similaire à la démonstration du théorème~\ref{th-h0gl}, en s'appuyant cette fois-ci sur le corollaire~\ref{cor-auxhg}. On choisit pour $\Phi$ le foncteur identité et pour $\B$ la sous-catégorie de $\A^T$ formée des morphismes $f : A\to B$ qui induisent un morphisme $A/{\rm Rad}\,A\to B/{\rm Rad}\,B$ dont l'image par $\pi$ est un monomorphisme scindé, de sorte que les deux premières hypothèses de la proposition~\ref{pr-auxhg} sont trivialement vérifiées.

Quitte à quotienter par le radical, il suffit de vérifier les deux dernières hypothèses de la proposition~\ref{pr-auxhg} sur les objets $V$ appartenant à $\mathbf{Q}_{++}$ (de sorte que tous les objets de $\K(V)$ ont une image par $\pi$ qui est dans $\mathbf{M}(\A)$).

Considérons deux objets $V\xrightarrow{i} C$ et $V\xrightarrow{u} B$ de $\K(V)$ et montrons qu'ils appartiennent à la même composante connexe. Par le lemme~\ref{lm-ccd} (dont on conserve les notations), il suffit d'établir l'existence d'une application linéaire $f : \pi(B)\to\pi(C)$ telle que  $f\circ u=i$ et $\beta(u(a),b)=\gamma(i(a),f(b))$. Soient $p$ et $r$ des rétractions de $\pi(i)$ et $\pi(u)$ respectivement : il existe une (et une seule) application linéaire $\varphi : \pi(B)\to\pi(C)$ telle que $\gamma(x,\varphi(b))=\beta(up(x),b)$ ($\gamma$ est non dégénérée); posons $f=ir+\varphi(1-ur)$. Comme $r\circ u=1$, $f\circ u=i$, et
$$\gamma(i(a),f(b))=\gamma(i(a),i(r(b)))+\gamma(i(a),\varphi(b-ur(b)))=\dots$$
$$\beta(u(a),u(r(b)))+\beta(upi(a),b-ur(b))=\beta(u(a),b)$$
puisque $i$ et $u$ sont quadratiques et que $pi=1$. Cela montre la connexité de $\K(V)$.

La dernière propriété à vérifier se prouve de façon analogue : choisissons un morphisme $u : (V,q)\to (E,Q)$, où $(V,q)$ est toujours dans $\mathbf{Q}_{++}$, $(E,Q)$ dans $\mathbf{H}^T(\A)$ et $u$ possède une rétraction $A$-linéaire $p$. On applique le lemme~\ref{lm-ccd} au diagramme de $\mathbf{Q}_{++}$
$$\xymatrix{(V,2q)\ar[r]^-{diag}\ar[rd]_-u & (V,q)\oplus (V,q)\ar[r]^-{id\oplus u} & (V,q)\oplus (E,Q)\\
& (E,2Q) &
}$$
L'application linéaire $f : E\to V\oplus E$ de composantes $p$ et $2-up$ a bien une composée avec $u$ de  composantes $1$ et $u$, et l'on a, pour $t\in V$ et $x\in E$,
$$<(t,u(t)),f(x)>_{q\oplus Q}=<t,p(x)>_q+<u(t),2x-up(x)>_Q=2<u(t),x>_Q$$
comme souhaité, puisque $u$ est quadratique (on désigne ici par $<-,->_a$ la forme bilinéaire associée à une forme quadratique $a$). Cela termine la démonstration.
\end{proof}

\begin{cor}\label{cor-ar}
 Soient $F\in {\rm Ob}\,\F(A)$ un foncteur analytique et $n\in\mathbb{N}$. Il existe un isomorphisme
$$H_n(O_\infty(A);F_\infty)\simeq\bigoplus_{i+j=n}H_i\big(O_\infty(A);{\rm Tor}_j^{\mathbf{P}(A)}(\mathbb{Z}[S^2]_A^\vee,F)\big).$$
\end{cor}

\appendix
\section{Homologie des foncteurs}

Cet appendice rappelle quelques définitions et propriétés homologiques classiques des catégories de foncteurs (d'une petite catégorie vers les groupes abéliens). On pourra se reporter par exemple à l'article de Pirashvili dans l'ouvrage de synthèse \cite{FFPS} pour plus de détails et de références.

\subsection{Généralités}\label{sagl}

Dans ce paragraphe, $\C$ désigne une petite catégorie.

La catégorie $\C-\mathbf{Mod}$ est une catégorie de Grothendieck avec un ensemble de générateurs projectifs de type fini $P^\C_c:=\mathbb{Z}[\C(c,-)]$ --- le lemme de Yoneda procure en effet un isomorphisme canonique $(\C-\mathbf{Mod})(P^\C_c,F)\simeq F(c)$.

On dispose d'un bifoncteur biadditif
$$\underset{\C}{\otimes} : (\mathbf{Mod}-\C)\times (\C-\mathbf{Mod})\to\mathbf{Ab}$$
tel qu'existent des isomorphismes canoniques
$$P^{\C^{op}}_c\underset{\C}{\otimes}F\simeq F(c)\quad\text{et}\quad G\underset{\C}{\otimes}P^\C_c\simeq G(c)$$
($F\underset{\C}{\otimes}G$ est la {\em cofin} --- cf. \cite{ML-cat}, chap.~IX, §\,6 --- du bifoncteur $F\boxtimes G : \C^{op}\times\C\to\mathbf{Ab}$).

Le bifoncteur $\underset{\C}{\otimes}$ est exact à droite relativement à chacune des variables et on peut le
dériver comme un produit tensoriel usuel (cf. \cite{Wei}, §\,2.7, par exemple), obtenant un bifoncteur homologique (groupes de torsion sur $\C$)
${\rm Tor}^\C_*$ (ces groupes abéliens peuvent se décrire comme homologie d'un complexe explicite de type barre
construit sur le nerf de la catégorie $\C$). On dispose d'un
isomorphisme canonique ${\rm Tor}^\C_*(G,F)\simeq {\rm Tor}^{\C^{op}}_*(F,G)$ ; les foncteurs acycliques pour Tor sont dits plats. Les foncteurs Tor commutent aux colimites filtrantes en chaque variable.

On dispose d'une fonctorialité en $\C$ : supposons que $\varphi$ est un foncteur de $\C$ dans une autre petite
catégorie $\D$. Pour tous foncteurs $G\in {\rm Ob}\,\mathbf{Mod}-\D$ et $F\in {\rm Ob}\,\D-\mathbf{Mod}$, $\varphi$ induit
un morphisme naturel $\varphi_* : {\rm Tor}^\C_*(\varphi^*G,\varphi^*F)\to {\rm Tor}^\D_*(G,F)$ (on rappelle que $\varphi^*$
désigne la précomposition par $\varphi$).

L'observation suivante est d'usage courant : si $\varphi, \psi : \C\to\D$ sont des foncteurs et $u : \varphi\to\psi$
une transformation naturelle, le diagramme
\begin{equation}\label{eq-commutor}
\xymatrix{{\rm Tor}^\C_*(\psi^*G,\varphi^*F)\ar[rrr]^-{{\rm Tor}^\C_*(\psi^*G,F\circ u)}\ar[d]_-{{\rm Tor}^\C_*(G\circ u,\varphi^*F)} & & &
{\rm Tor}^\C_*(\psi^*G,\psi^*F)\ar[d]^{\psi_*}\\
{\rm Tor}^\C_*(\varphi^*G,\varphi^*F)\ar[rrr]^-{\varphi_*} & &  & {\rm Tor}^\D_*(G,F)
}
\end{equation}
commute pour tous $G\in {\rm Ob}\,\mathbf{Mod}-\D$ et $F\in {\rm Ob}\,\D-\mathbf{Mod}$.

L'homologie de $\C$ à coefficients dans un foncteur $F\in {\rm Ob}\,\C-\mathbf{Mod}$ est $H_*(\C;F):={\rm Tor}^\C_*(\mathbb{Z},F)$
(où $\mathbb{Z}$ désigne le foncteur constant de $\mathbf{Mod}-\C$). Lorsque le foncteur $F$ est lui-même constant,
cette homologie n'est autre que l'homologie du nerf (ou classifiant) de la catégorie $\C$ ; on renvoie à l'article
fondateur de Quillen \cite{QK} pour les résultats de base. On retiendra en particulier que l'existence d'une
transformation naturelle entre deux foncteurs (entre petites catégories) implique qu'ils induisent le même morphisme en homologie (ce qui se déduit du diagramme commutatif~(\ref{eq-commutor})).

\paragraph*{Extensions de Kan}
Étant donné un foncteur $\varphi : \C\to\D$, définissons $\varphi_! : \mathbf{Mod}-\C\to\mathbf{Mod}-\D$ par
\begin{equation}\label{eq-exk}
 \varphi_!(G)(d):=G\underset{\C}{\otimes}\varphi^*(P^\D_d).
\end{equation}

L'extension de Kan $\varphi_!$ satisfait à un isomorphisme naturel
\begin{equation}\label{eq-exks}
\varphi_!(G)\underset{\D}{\otimes}F\simeq G\underset{\C}{\otimes}\varphi^*(F)
\end{equation}
(c'est une variante d'adjonction).

Si le foncteur $\varphi_!$ est exact, cet isomorphisme s'étend aux groupes de torsion en un isomorphisme naturel
\begin{equation}\label{eqexkdf}
 {\rm Tor}^\D_*(\varphi_!(G),F)\simeq {\rm Tor}^\C_*(G,\varphi^*(F)).
\end{equation}

L'exactitude de $\varphi_!$ est en particulier vérifiée si $\varphi$ possède un adjoint à gauche $\psi$ : on a alors $\varphi^*(P^\D_d)\simeq P^\C_{\psi(d)}$, d'où $\varphi_!\simeq\psi^*$.

L'exactitude de $\varphi_!$ vaut également dans la situation suivante : supposons que $T$ est un foncteur contravariant de $\D$ vers les ensembles, $\C=\D^T$ (cf. notation~\ref{nice} en fin d'introduction) et $\varphi=\pi^T$. On voit aussitôt que $\varphi_!$ est donné par
\begin{equation}\label{adj-gen}
\varphi_!(G)(d)=\bigoplus_{x\in T(d)} G(d,x) 
\end{equation}

Dans le cas général, $\varphi_!$ est un fonteur exact à droite (il commute même à toutes les colimites) qui préserve
les objets projectifs ($\varphi_!(P^{{\C}^{op}}_c)\simeq P^{\D^{op}}_{\varphi(c)}$), de sorte que l'isomorphisme~(\ref{eq-exks})
se dérive en une suite spectrale de Grothendieck (cf. \cite{Wei}, corollaire 5.8.4, par exemple) du type :
\begin{equation}\label{eq-ssgk}
 E^2_{i,j}={\rm Tor}^\D_i((\mathbb{L}_j\varphi_!)(G),F)\Rightarrow {\rm Tor}_{i+j}^\C(G,\varphi^*(F))
\end{equation}
où les foncteurs dérivés à gauche $\mathbb{L}_j\varphi_!$ sont donnés par
\begin{equation}\label{eq-exkd}
 (\mathbb{L}_*\varphi_!)(G)(d):={\rm Tor}^\C_*(G,\varphi^*(P^\D_d)).
\end{equation}

Si l'on note
\begin{equation}\label{eq-kum}
\mathcal{L}_*^\varphi(X)=Ker\,((\mathbb{L}_*\varphi_!)(\varphi^* X)\to X),
\end{equation}
pour $X\in {\rm Ob}\,\mathbf{Mod}-\D$,
le noyau de l'unité, on a donc :
\begin{pr}\label{pr-anek}
 Supposons que $X\in {\rm Ob}\,\mathbf{Mod}-\D$ et $F\in {\rm Ob}\,\D-\mathbf{Mod}$ sont tels que
${\rm Tor}^\D_i(\mathcal{L}_j^\varphi(X),F)=0$ pour tous entiers $i$ et $j$. Alors le morphisme
$${\rm Tor}^\C_*(\varphi^*(X),\varphi^*(F))\to {\rm Tor}^\D_*(X,F)$$
qu'induit $\varphi$ est un isomorphisme.
\end{pr}

\begin{rem}\label{rq-ekz}
 Le foncteur $\mathcal{L}^\varphi_*(\mathbb{Z})$ est donné par un isomorphisme canonique
$$\mathcal{L}^\varphi_*(\mathbb{Z})(d)\simeq\tilde{H}_*(\C_{\varphi^*\D(d,-)})$$
où $\tilde{H}_*$ désigne l'homologie réduite (noyau de la flèche vers $\mathbb{Z}$ induite par l'unique foncteur vers la catégorie à un morphisme). Cela découle des considérations précédentes appliquées au foncteur $\pi_{\varphi^*\D(d,-)} : \C_{\varphi^*\D(d,-)}\to\C$.
\end{rem}

\paragraph*{Produits tensoriels extérieurs} Le produit tensoriel extérieur de $F\in {\rm Ob}\,\C-\mathbf{Mod}$ et $G\in {\rm Ob}\,\D-\mathbf{Mod}$ est le foncteur $F\boxtimes G\in {\rm Ob}\,(\C\times\D)-\mathbf{Mod}$ donné par $(F\boxtimes G)(c,d)=F(c)\otimes G(d)$. La propriété suivante est classique et utile :
\begin{pr}\label{pr-bifo}
Soient $\B$ et $\C$ deux petites cat\'egories, $X$ et $Y$ des objets de $\mathbf{Mod}-\B$ et $\mathbf{Mod}-\C$
respectivement et $E$ un objet de $(\B\times\C)-\mathbf{Mod}$. On suppose que $Y$ prend des valeurs plates sur $\mathbb{Z}$. Il existe alors une suite spectrale
$$E^2_{i,j}={\rm Tor}^\B_{i}\big(X,U\mapsto {\rm Tor}^\C_j(Y,E(U,-))\big)\Rightarrow {\rm Tor}^{\B\times\C}_{i+j}(X\boxtimes Y,E).$$
\end{pr}

\begin{proof}
La d\'efinition du produit tensoriel au-dessus d'une petite cat\'egorie procure aussitôt un isomorphisme canonique
$$(X\boxtimes Y)\underset{\B\times\C}{\otimes} E\simeq X\underset{\B}{\otimes}\big(U\mapsto Y\underset{\C}{\otimes}E(U,-)\big).$$
L'hypothèse sur $Y$ assure que $U\mapsto Y\underset{\C}{\otimes}E(U,-)$ est un foncteur plat si $E$ est projectif. On peut donc dériver le membre de droite (vu comme foncteur en $E$) en une suite spectrale de Grothendieck, qui prend la forme annonc\'ee.
\end{proof}

\paragraph*{Homologie de Hochschild} Si $B$ est un bifoncteur sur $\C$, c'est-à-dire un objet de $(\C^{op}\times\C)-\mathbf{Mod}$, on peut définir l'homologie de Hochschild de $\C$ à coefficients dans $B$ comme
$$HH_*(\C;B):={\rm Tor}_*^{\C^{op}\times\C}(\mathbb{Z}[\C^{op}],B)$$
(où $\mathbb{Z}[\C^{op}]$ est le foncteur contravariant $(a,b)\mapsto\mathbb{Z}[\C(b,a)]$ sur $\C^{op}\times\C$).

On peut aussi exprimer ces groupes comme de l'homologie ordinaire de la {\em catégorie des factorisations} $\mathbf{F}(\C):=(\C^{op}\times\C)^{\C^{op}}$ (introduite par Quillen dans \cite{QK}) --- explicitement, ses objets sont les flèches $f : b\to a$ de $\C$, et les morphismes de $f$ vers $g : d\to c$ sont les diagrammes commutatifs
$$\xymatrix{b\ar[r]^-f\ar[d] & a \\
d\ar[r]^-g & c\ar[u]
}$$
de $\C$. En effet, les isomorphismes~(\ref{adj-gen}) et~(\ref{eqexkdf}) se spécialisent en
\begin{equation}\label{eq-hofac}
 HH_*(\C;B)\simeq H_*(\mathbf{F}(\C);\pi^*B)
\end{equation}
où l'on a noté simplement $\pi$ pour $\pi^{\C^{op}}$.

On peut également définir $HH_0(\C;B)$ comme la cofin de $B$ et les $HH_*$ comme les foncteurs dérivés correspondants ; on en déduit un isomorphisme canonique $HH_0(\C;F\boxtimes G)\simeq F\underset{\C}{\otimes} G$ qui s'étend en un isomorphisme gradué naturel $HH_*(\C;F\boxtimes G)\simeq {\rm Tor}^\C_*(F,G)$ lorsqu'un des foncteurs $F$ et $G$ prend des valeurs plates sur $\mathbb{Z}$ (le lien avec la définition donnée plus haut est fait par la proposition~\ref{pr-bifo}). 

%

\subsection{Quelques propriétés des catégories $\F(A)$}\label{app2}

Si $A$ est un anneau, on note $\mathbf{P}(A)$ la catégorie des $A$-modules à gauche projectifs de type fini et $\F(A)=\mathbf{P}(A)-\mathbf{Mod}$ (pour être parfaitement correct, il faut remplacer ici $\mathbf{P}(A)$ par un squelette).

La situation qui a été la plus étudiée est celle où $A$ est un corps fini (voir par exemple les calculs importants de \cite{FFSS}) ; nous nous intéressons plutôt ici à des anneaux sans torsion (comme groupes abéliens).

\begin{pr}[Changement plat d'anneau]\label{pr-chplat}
 Soit $\varphi : A\to B$ un morphisme {\rm plat} d'anneaux. Par abus, on note encore $\varphi : \mathbf{P}(A)\to\mathbf{P}(B)$ le foncteur de changement de base $B\underset{A}{\otimes}-$.
\begin{enumerate}
 \item L'extension de Kan $\varphi_! : \F(A)\to\F(B)$ définie par 
$$\varphi_!(F)(M)=\varphi^* P^{\mathbf{P}(B)^{op}}_M\underset{\mathbf{P}(A)}{\otimes}F$$
(elle donc {\rm duale} de celle donnée par (\ref{eq-exk})) est exacte ; elle induit un isomorphisme naturel
$${\rm Tor}^{\mathbf{P}(B)}_*(G,\varphi_! F)\simeq {\rm Tor}^{\mathbf{P}(A)}_*(\varphi^*G,F).$$
\item Soit $F$ un endofoncteur de la catégorie des groupes abéliens commutant aux colimites filtrantes. L'image par $\varphi_!$ de la restriction de $F$ à $\mathbf{P}(A)$ est la restriction de $F$ à $\mathbf{P}(B)$.
\item Si $X : \mathbf{P}(B)^{op}\times\mathbf{P}(B)\to\mathbf{Ab}$ est un bifoncteur qui est la restriction d'un foncteur $Y : \mathbf{P}(B)^{op}\times\mathbf{Ab}\to\mathbf{Ab}$ commutant aux colimites filtrantes en la variable covariante, il existe un isomorphisme naturel
$$HH_*(\mathbf{P}(B);X)\simeq HH_*(\mathbf{P}(A);\tilde{X})$$
où $\tilde{X}$ est la précomposition de $Y$ par
$$\mathbf{P}(A)^{op}\times\mathbf{P}(A)\to\mathbf{P}(B)^{op}\times\mathbf{Ab}\quad (M,N)\mapsto (B\underset{A}{\otimes}M,N).$$
\end{enumerate}
\end{pr}

\begin{proof}
 Les deux premières assertions découlent de l'adjonction entre restriction et extension des scalaires et de ce qu'un module plat est colimite {\em filtrante} de modules projectifs de type fini (ainsi que de l'isomorphisme (\ref{eqexkdf}), dans sa version duale). La dernière s'en déduit lorsque $Y$ est un produit tensoriel de deux foncteurs dont l'un prend des valeurs plates sur $\mathbb{Z}$, on obtient le cas général par des arguments usuels d'algèbre homologique (les projectifs standard sur une catégorie produit sont des produits tensoriels extérieurs à valeurs $\mathbb{Z}$-plates).
\end{proof}

On note ${\rm I}_A\in {\rm Ob}\,\F(A)$ le foncteur d'inclusion $\mathbf{P}(A)\hookrightarrow\mathbf{Ab}$. Si $A$ est muni d'une anti-involution, on note par ailleurs $(-)^\vee$ le foncteur de précomposition par la dualité ${\rm Hom}_A(-,A) : \mathbf{P}(A)^{op}\to\mathbf{P}(A^{op})\simeq\mathbf{P}(A)$ (le dernier isomorphisme étant donné par l'anti-involution de $A$).

\begin{cor}\label{cor-chpz}
 Si $A$ est un anneau $\mathbb{Z}$-plat muni d'une anti-involution, on dispose d'un isomorphisme naturel
$${\rm Tor}^{\mathbf{P}(A)}_*({\rm I}_A^\vee,{\rm I}_A)\simeq A\otimes {\rm Tor}^{\mathbf{P}(\mathbb{Z})}_*({\rm I}_\mathbb{Z}^\vee,{\rm I}_\mathbb{Z}).$$
\end{cor}

Le groupe de torsion apparaissant à droite est connu :

\begin{thm}[Franjou-Pirashvili]\label{th-fp}
 Le groupe abélien ${\rm Tor}_i^{\mathbf{P}(\mathbb{Z})}({\rm I}_\mathbb{Z}^\vee,{\rm I}_\mathbb{Z})$ est :
\begin{enumerate}
 \item isomorphe à $\mathbb{Z}$ pour $i=0$ ;
\item nul pour $i>0$ pair ;
\item cyclique d'ordre $n$ pour $i=2n-1$.
\end{enumerate}
\end{thm}

\begin{proof}
 Dans \cite{FPZ}, Franjou et Pirashvili calculent ${\rm Ext}^*_{\F(\mathbb{Z})}({\rm I}_\mathbb{Z},{\rm I}_\mathbb{Z})$. Le calcul de groupe de torsion en question s'en déduit par la suite exacte des coefficients universels, en notant que ces groupes de torsion sont nécessairement de type fini.
\end{proof}

\bibliographystyle{smfalpha}
\bibliography{bibli-hostu.bib}
\end{document}